\documentclass{amsart}

\usepackage{amsmath}
\usepackage{amsthm}
\usepackage{enumerate}
\usepackage[shortlabels]{enumitem}
\usepackage{amssymb}
\usepackage[utf8]{inputenc}
\usepackage[colorlinks  = true,
            citecolor   = blue,
            linkcolor   = blue,
            bookmarks   = true,
            linktocpage = true]{hyperref}
\usepackage[all]{xy}
\usepackage{tikz-cd}
\usepackage[textsize=scriptsize,color=yellow]{todonotes}

\title{A Bivariant Theory for the Cuntz Semigroup}

\author{Joan Bosa}
\address{School of Mathematics and Statistics, University of Glasgow, 15 University Gardens, G12 8QW, Glasgow, UK}
\email{joan.bosa@glasgow.ac.uk}
\email{g.tornetta.1@research.gla.ac.uk}
\email{joachim.zacharias@glasgow.ac.uk}

\author{Gabriele Tornetta}

\author{Joachim Zacharias}

\thanks{\emph{Supported by:}  EPSRC Grant EP/I019227/2. The first author holds a Beatriu de Pinos Fellowship (BP-A 00123)}

\date{\today}

\newcommand{\A}[1]{#1}
\newcommand{\Csalg}{C$^*$-algebra}
\newcommand{\Cs}{C$^*$}

\newcommand{\bb}[1]{\mathbb{#1}}
\newcommand{\IN}{\bb N}

\newcommand{\CC}{\bb C}

\newcommand{\dd}{^{**}}

\newcommand{\id}{\operatorname{id}}
\newcommand{\Id}{\operatorname{Id}}
\newcommand{\hoplus}{\operatorname{{\hat{\oplus}}}}
\newcommand{\SOT}{\text{\normalfont\scshape sot}}
\newcommand{\supp}{\operatorname{supp}}
\newcommand{\rk}{\operatorname{rk}}
\newcommand{\ev}{\operatorname{ev}}
\newcommand{\Lf}{\operatorname{Lf}}
\newcommand{\Mf}{\operatorname{Mf}}

\newcommand{\Ad}{\operatorname{Ad}}

\newcommand{\cc}{\subset\hskip-.5em\subset}

\newcommand{\cue}{\sim_{\text{Cu}}}

\newcommand{\au}{\approx_{\text{a.u.}}}

\newcommand{\cat}[1]{\text{\normalfont\sffamily #1}}

\newcommand{\mor}[1]{\operatorname{#1}}

\newcommand{\Hom}{\mor{Hom}}

\newcommand{\Cu}{\cat{Cu}}
\newcommand{\op}{^{\operatorname{op}}}
\newcommand{\loc}{_{\text{\sffamily loc}}}

\newcommand{\seq}[1]{\left\{{#1}_n\right\}_{n\in\IN}}
\newcommand{\norm}[1]{\left\Vert{#1}\right\Vert}

\newcommand{\eps}{\epsilon}
\newcommand{\Mi}{{M_\infty}}
\newcommand{\ess}{\text{\upshape ess}}
\newcommand{\Tau}{\text{\upshape T}}

\theoremstyle{plain}
\newtheorem{lemma}{Lemma}[section]
\newtheorem{theorem}[lemma]{Theorem}
\newtheorem{corollary}[lemma]{Corollary}
\newtheorem{proposition}[lemma]{Proposition}
\newtheorem*{proposition*}{Proposition}
\newtheorem*{theorem*}{Theorem}
\newtheorem*{definition*}{Definition}
\newtheorem*{claim*}{Claim}

\theoremstyle{definition}
\newtheorem{definition}[lemma]{Definition}
\newtheorem{example}[lemma]{Example}

\theoremstyle{remark}
\newtheorem{remark}[lemma]{Remark}

\newcommand{\e}{\epsilon}
\renewcommand{\_}{{\ \cdot\ }}

\newcommand{\sth}{\cite[Theorem 2.3]{wz2009}}

\begin{document}

	\maketitle
	\begin{abstract}
We introduce a bivariant version  of the Cuntz
semigroup  as equivalence classes of order zero maps generalizing the ordinary Cuntz semigroup. The theory has many 
properties formally analogous to KK-theory including a composition product.
We establish basic properties, like additivity, stability and continuity, 
and study categorical aspects in the setting of local \Cs-algebras. 
We determine the bivariant Cuntz semigroup for numerous examples such as when the second algebra is a Kirchberg algebra, and Cuntz homology for compact 
Hausdorff spaces which provides a complete invariant. 
Moreover, we establish identities when tensoring with  strongly
self-absorbing \Cs-algebras. Finally, we show
that the bivariant Cuntz semigroup of the present work can be used to classify
all unital and stably finite \Cs-algebras.

	\end{abstract}

	\setcounter{tocdepth}{1}
	\tableofcontents
\section{Introduction}

The Cuntz semigroup was introduced in the pioneering work \cite{cuntz78} of Joachim Cuntz in the 1970's as a \Cs-analogue of the Murray-von Neumann semigroup of projections in von Neumann algebras, 
replacing equivalence classes of projections by suitable equivalence classes of positive elements in the union of all matrix iterations of the algebra.  In a suitable sense, it may be regarded as a refinement of the $K_0$-group, since projections are particular positive elements. It is designed as a tool for studying algebras with traces, and for those algebras it is typically a very rich object (see \cite{ce} for further details). 
Therefore, it is not surprising that there has been a renewed and growing recent interest in the Cuntz semigroup as a classification tool
within Elliott's Classification Programme which aims at classifying simple nuclear \Cs-algebras using K-theoretic and trace space data, information which the Cuntz semigroup contains. 
Moreover, Toms' famous examples of non-isomorphic simple nuclear \Cs-algebras with the same Elliott invariant can be distinguished by their Cuntz semigroups \cite{toms}.
In a subsequent work, Perera and Toms formulated a classification conjecture in the setting of the much larger class of simple
exact \Cs-algebras using the Cuntz semigroup \cite{pereratoms}. Currently, no counterexamples to this conjecture are known. Thus, the Cuntz semigroup is a strong potential tool for classification, within and beyond the Elliott Programme.

In recent years research on the Cuntz semigroup intensified at different levels (see e.g. \cite{ABPP14, ERS11, robert}). In \cite{cei2008}, the lack of continuity of the original Cuntz semigroup functor $W(\_)$ is resolved by using its stabilized version and changing the target category slightly. More precisely, a richer category of abstract monoids, called $\Cu$, is defined, and it is shown that all Cuntz semigroups of stable algebras $\Cu(A)=W(A\otimes K)$ belong to it. Note that the role of stabilization and the Cuntz semigroup was clarified (\cite{ABP11}) defining a ``completion'' on $W(A)$. In this new setting, the functor $\Cu(\_)$ is sequentially continuous, which is shown by an alternative description as classes of Hilbert modules. 
Expanding on this line of research, Antoine, Perera and Thiel have recently shown in \cite{apt2014} that the functor $W(\_)$ is continuous under arbitrary inductive limits, provided that both, the domain and codomain categories are suitably chosen. In particular, they consider the category of {\it local}\Cs-algebras as the domain and a new category, called $\cat W$, as codomain.

The main purpose of the present paper is to introduce and study a bivariant version of the Cuntz semigroup. Our main motivation stems from Kasparov's bivariant KK-theory. This theory contains K-theory and K-homology as special cases. Moreover, the Kasparov product is a central and very powerful tool, leading together with the great flexibility of KK-theory to various far reaching applications in \Cs-algebra theory as well as geometry and topology. KK-theory is also of key importance in the Classification Programme. For instance, the Kirchberg-Phillips classification for purely infinite \Cs-algebras can be regarded as an entirely KK-theoretical result asserting that two Kirchberg algebras are isomorphic if and only if they are KK-equivalent, i.e. there exists an invertible element in the corresponding KK-group (\cite{KP002,KP00}).
It can be reasonably expected that the bivariant Cuntz semigroup should play a similar role for 
the classification of stably finite \Cs-algebras and indeed we provide a classification theorem along those lines in section 6 where we prove that unital stably finite \Cs-algebras are isomorphic if and only if there exists a (strictly) invertible element in the bivariant Cuntz semigroup. 
This implies the known fact that the ordinary Cuntz semigroup provides a complete invariant for a certain class of \Cs-algebras such as AI-algebras and  inductive limits of one-dimensional non-commutative CW complexes \cite{robert}. 
  
The definition we propose is based on equivalence classes of completely positive maps of order zero, c.p.c. order zero for short. These maps, which could also be called orthogonality preserving, were introduced and fully characterized in \cite{wz2009}, by Winter and the third named author based on previous related concepts. For instance $*$-homomorphisms are c.p.c. order zero maps.  It is shown in \cite{wz2009} that this class of maps induce maps between the Cuntz semigroups and trace spaces of the corresponding algebras (cf. \cite[Corollary 3.5]{wz2009}). 
Just like in KK-theory, the resulting bivariant Cuntz semigroup contains the ordinary Cuntz semigroup by specializing the first variable to $\CC$. Specializing the second variable to $\CC$ leads to a contravariant functor which we term \emph{Cuntz homology}.

Our bivariant theory admits  a product given essentially by composition and can be regarded as a refinement of KK-theory: the classes of projections in the ordinary Cuntz semigroup correspond in the bivariant setting to classes of $*$-homomorphisms from the first algebra to the stabilization of the second. The subsemigroup given by those appears as an unstable version of KK-theory. Moreover,  we show some functorial properties for the bivariant Cuntz semigroup, analogous to the ones that the  ordinary  Cuntz semigroup possesses, and  establish further properties analogous to those of KK-theory, like additivity, functoriality, continuity etc. Next we investigate which categorical setting is most
appropriate for our theory. To this end, we show that every bivariant Cuntz semigroup for a pair of \emph{local} \Cs-algebras is an object of the category $\cat W$ considered in \cite{apt2014}. Continuity, however,
appears to remain a special feature of the ordinary  Cuntz semigroup, since the bifunctor that emerges from our setting seems to be continuous only in very special cases. This, indeed, opens up for a Cuntz analogue of the notion of
KK-semiprojectivity of \cite{dadarlat}, but this aspect is not addressed in the present paper.

Alongside with the analogue of the ordinary definition of the Cuntz semigroup we
also give a bivariant extension of the Hilbert module picture described in \cite{cei2008}. This has the bearings of Kasparov cycles for KK-theory, but with a more
suitable set of axioms to accommodate the different nature of the
equivalence relation. It is then a natural question whether the bivariant object
that arises this way is an element of the category $\Cu$. To answer this
question affirmatively we would have to show, among other properties, that the stabilized
bivariant Cuntz semigroup is closed under suprema. We do not study this question in these notes, but a possible approach to it could be provided by a generalization of the open projection picture
for the Cuntz semigroup described in \cite{ort, btzop}.

The explicit computation of a Cuntz semigroup turns out to be a rather hard
problem, and the bivariant theory discussed here, being an extension of the
former, appears even harder to determine. There are two special cases though where such a computation
can be done by exploiting some fundamental results: when the codomain is a Kirchberg algebra 
and the domain is an exact \Cs-algebra, and when the first algebra is commutative and the second is just the set of scalars $\CC$ (Cuntz homology).
Moreover, we prove an absorption result for  strongly self-absorbing \Cs-algebras which is very useful in obtaining general identities for the bivariant Cuntz semigroup.


\subsection*{Outline of the paper} The present work is organized as follows. \textbf{Section 1} provides an introduction to the motivations behind this paper
and a quick overview of the results, alongside this outline and a list
of notation.

In \textbf{Section 2} we give the main definitions that constitute the bivariant
theory of the Cuntz semigroup that we are  presenting. In order to do so we
need to extend some well-known results concerning c.p.c. order zero maps to the
setting of local \Cs-algebras in the sense of \cite{apt2014}, together with some
other technical results that are used throughout. These are employed to
investigate the properties of additivity, functoriality and stability of the
bifunctor introduced within this section. Following the lead of
\cite{cei2008} we also define a stabilized version of the bivariant Cuntz
semigroup together with an equivalent module picture that closely resembles
Kasparov's formulation of KK-theory.

\textbf{Section 3} is  devoted to some further categorical aspects. We
investigate whether the bivariant Cuntz semigroup as defined in the previous
section is an object of the category $\cat W$ described in \cite{apt2014} and
study any possible continuity properties. 

In \textbf{Section 4} we introduce the analogue of Kasparov's product for the
bivariant Cuntz semigroup together with the resulting notion of Cuntz equivalence
between \Cs-algebras.

In  \textbf{Section 5}
we determine our explicit examples of bivariant Cuntz semigroups.
We show that if the first variable is an exact \Cs-algebra and the second is a Kirchberg algebra 
then the bivariant Cuntz semigroup  is isomorphic to the two-sided ideal
lattice of the algebra in the first argument. This generalizes the well-known
result that the Cuntz semigroup of a Kirchberg algebra is $\{0,\infty\}$, i.e.
the two-sided ideal lattice of $\CC$.
Next we consider the class of strongly
self-absorbing \Cs-algebras, which were studied systematically in \cite{tw} and
play an important r\^ole in the Classification Programme. We study the behavior of the bivariant Cuntz semigroup when
its arguments are tensored by such \Cs-algebras. An isomorphism theorem then allows
to identify certain bivariant Cuntz semigroups with ordinary ones. Finally, using spectral theory we obtain a description of
Cuntz homology for compact
metrizable Hausdorff spaces. This object turns out to be a complete invariant
for such spaces.

In \textbf{Section 6} we use the composition product and the notion of strict Cuntz equivalence, which involves a scale condition, to provide a classification result for the bivariant Cuntz semigroup. Specifically, we show that any two unital and stably finite
\Cs-algebras are isomorphic if and only if they are \emph{strictly}
Cuntz-equivalent.


\subsection*{Notation} In this paper we have employed standard notation whenever possible. However, for
the reader's convenience, we provide a brief summary of some possibly
non-standard notation that has been employed throughout the present work.

\begin{enumerate}[align=parleft]
\item[$\Delta$]
Diagonal map $\Delta : A \to A\oplus A$. It embeds $A$ into $A\oplus
A$ \emph{diagonally} by $\Delta(a) := a\oplus a$, for any $a\in A$.

\item[$\oplus$]
Direct sum. For two maps $\phi:A\to B$ and $\psi:C\to
D$ we have $\phi\oplus\psi :A\oplus C\to B\oplus D$ defined as
$(\phi\oplus\psi)(a\oplus c) := \phi(a)\oplus\psi(c)$.

\item[$\hoplus$]
Direct sum of maps precomposed with $\Delta$. For maps
$\phi:A\to B$ and $\psi:A\to C$, $\phi\hoplus\psi:A\to B\oplus C$
is given by $(\phi\oplus\psi)\circ\Delta$.

\item[$\Subset$]
Finite subset of.

\item[$\phi^{(n)}$]
$n$-th ampliation of the map $\phi$. That is, $\phi\otimes\id_{M_n}$ for any $n\in\IN\cup\{\infty\}$.

\item[$K$] The \Cs-algebra of compact operators on a infinite-dimensional separable Hilbert space.

\end{enumerate}

\subsection*{Acknowledgement} The authors would like to thank Francesc Perera, Hannes Thiel, Stuart White and  Wilhelm Winter as well as many other colleagues for many helpful discussions. We are particularly grateful to Wilhelm for providing us with some unpublished notes on the subject of the present article.


\section{Definitions and Properties}

In this work we make use of the notion of local \Cs-algebras in the sense
of \cite{apt2014}, i.e. a pre-\Cs-algebra $A$ is \emph{local} if there is an
arbitrary family of \Cs-subalgebras $\{A_i\}_{i\in I}$ of $A$ with the property
that for any $i,j\in I$ there is $k\in I$ such that $A_i,A_j\subset A_k$ and
$A=\bigcup_{i\in I} A_i$. Equivalently a \Cs-algebra $A$ is local if the
\Cs-algebra generated by a finite subset $F$ of $A$ in the completion of $A$ is
contained in $A$. In particular, every local \Cs-algebra in this sense is closed
under continuous functional calculus of its normal elements. The reason why we
want to consider such a structure is because we make use of the infinite matrix
algebra $M_\infty(A)$ over a \Cs-algebra $A$ throughout, which is a typical
example of a local \Cs-algebra.

At this point we make the blanket assumption that all the \Cs-algebras we
consider are separable, unless otherwise stated.

\begin{proposition}\label{prop:cpclocal} Let $A$ and $B$ be local \Cs-algebras,
$\tilde A$ and $\tilde B$ their respective completions, and $\phi:A\to B$ be a
c.p.c. order zero map. Then there exists a unique c.p.c. order zero extension
$\tilde\phi:\tilde A\to\tilde B$ of $\phi$.
\end{proposition}
\begin{proof} Let $\tilde\phi$ be the c.p.c. extension of $\phi$ to the
completions. One needs to check that orthogonality of elements in the completion
is preserved by $\tilde\phi$. For any pair of positive contractions
$a,b\in\tilde A^+$ with the property $ab = 0$, take sequences $\seq a,\seq
b\subset A_1$, the unit ball of $A$, with the property that $a_n\to a$ and
$b_n\to b$. Because $A$ is a local \Cs-algebra, the \Cs-algebras $A_n$ generated
by $\{a_n,b_n\}$ inside $\tilde A$ are contained in $A$; hence, one can consider
the restrictions $\phi_n:=\phi|_{A_n}$. These are c.p.c. order zero maps over
\Cs-algebras. By the structure theorem \sth , there are positive contractions
$h_n$ and $*$-homomorphisms $\pi_n$ such that $\phi_n(a) = h_n\pi_n(a)$ for any
$a\in A_n$ and $n\in\IN$. By construction $\tilde\phi$ extends each $\phi_n$, so one has the identity
	$$\tilde\phi(a_n)\tilde\phi(b_n) = h_n\tilde\phi(a_nb_n).$$
By the joint continuity of the norm and the boundedness of $\tilde\phi$, one then gets
	$$\tilde\phi(a)\tilde\phi(b) = 0,$$
which shows that the extension $\tilde\phi$ preserves orthogonality of positive
elements and, therefore, has the order zero property.
\end{proof}

\begin{remark} We sometimes refer to the $*$-homomorphism arising from a
c.p.c. order zero map $\phi$ through \sth\ as the \emph{support
$*$-homo\-mor\-phism} of $\phi$. It will usually be denoted by $\pi_\phi$. Note that in a representation $\phi (a)  = h \pi (a)$, the element $h$ and the homomorphism $\pi$ are not always unique but they are uniquely determined by requiring that the support projections of $\phi$, $\pi$ and  $h$  are all equal. Often we will assume that tacitly.
\end{remark}

From the above result one can conclude that the structure theorem \cite[Theorem
2.3]{wz2009} generalizes to c.p.c. order zero maps between local \Cs-algebras,
as shown by the following

\begin{corollary}\label{cor:structure} Let $A$ and $B$ be local \Cs-algebras,
and let $\phi:A\to B$ be a c.p.c. order zero map. There is a positive
element $h\in \mathcal M(C^*(\phi(A)))$ and a $*$-homomorphism $\pi:A\to\mathcal
M(C^*(\phi(A)))\cap\{h\}'$ such that $\norm\phi = \norm h$ and $\phi(a) =
h\pi(a)$ for any $a\in A$.
\end{corollary}
\begin{proof} It suffices to apply \cite[Theorem 2.3]{wz2009} to the extension
provided by Proposition \ref{prop:cpclocal} and restrict the support
$*$-homomorphism $\pi_{\tilde\phi}$ to $A$ in order to obtain the sought
$*$-homomorphism $\pi$.
\end{proof}

Following the ideas given in the introduction, we define a notion of comparison
of c.p.c. order zero maps. First of all we record the following result, whose
proof is routine and therefore omitted.

\begin{proposition}\label{prop:subeq} Let $A$ and $B$ be local \Cs-algebras and
$\phi,\psi:A\to B$ be two c.p.c. order zero maps. The following are equivalent.
	\begin{enumerate}[i.]
		\item $\exists\seq b\subset B\quad\text{such that}\quad \norm{b_n^*\psi(a)b_n-\phi(a)}\to0$
		for any $a\in A$;
		\item $\forall F\Subset A,\eps > 0\quad\exists\, b\in B\quad\text{such that}\quad
		\norm{b^*\psi(a)b-\phi(a)}<\eps$ for any $a\in F$.
	\end{enumerate}
\end{proposition}

\begin{definition} Let $A$ and $B$ be local \Cs-algebras and $\phi,\psi:A\to B$
be c.p.c. order zero maps. Then $\phi$ is said to be (Cuntz-)subequivalent to
$\psi$, in symbols $\phi\precsim\psi$, if any of the conditions of Proposition
\ref{prop:subeq} holds.
\end{definition}

It is left to the reader to check that the above relation defines a pre-order
among c.p.c. order zero maps between local \Cs-algebras. The antisymmetrization
yields an equivalence relation $\sim$, that is $\phi\sim\psi$ if
$\phi\precsim\psi$ and $\psi\precsim\phi$.

\begin{definition}\label{def:bcu} Let $A$ and $B$ be local \Cs-algebras. The
bivariant Cuntz semigroup $W(A,B)$ of $A$ and $B$ is the set of equivalence
classes
		$$W(A,B) = \{\phi:A\to M_\infty(B)\ |\ \phi\text{ is c.p.c. order
		zero}\}/\sim,$$
	endowed with the binary operation $+:W(A,B)\times W(A,B)\to W(A,B)$ given by%
		$$[\phi] + [\psi] = [\phi\hoplus\psi].$$
\end{definition}

In \cite[Proposition 2.5]{apt2009} it is shown that Cuntz comparison of positive
elements from a commutative \Cs-algebra $A$ reduces to the containment relation
of the corresponding support projections. Specifically, if $X$ is a compact
Hausdorff space and $f,g\in C(X)$ are positive functions, then $f\precsim g$ if
and only if $\supp f\subseteq \supp g$. We now provide the analogue of this result
for the Cuntz comparison of c.p.c. order zero maps.

\begin{proposition} Let $X$ be a compact Hausdorff space, $A$ be a unital
\Cs-algebra and $\phi,\psi:A\to C(X)$ be c.p.c. order zero maps. Then,
$\phi\precsim\psi$ if and only if $\supp\phi(1)\subseteq\supp\psi(1)$ and
$\pi_\phi(a) = \pi_\psi(a)\chi_{\supp\phi(1)}$ for any $a\in A$.
\end{proposition}
\begin{proof} One implication is trivial, so let us focus on the converse. If $\phi\precsim\psi$, then there
exists $\seq f\subset C(X)$ such that
	$$\norm{|f_n|^2\psi(a)-\phi(a)}\to0\qquad\forall a\in A.$$
In particular, this holds for $a=1$, which implies
	$$\norm{|f_n|^2\psi(1)-\phi(1)}\to0\quad \iff
	\quad\phi(1)\precsim\psi(1)\quad\iff\quad\supp\phi(1)\subseteq\supp\psi(1).$$
By using \sth\ on $\phi$ and $\psi$, one has
	$$\norm{|f_n|^2\psi(1)\pi_\psi(a)-\phi(1)\pi_\phi(a)}\to0\qquad\forall a\in
	A,$$
which, together with the previous condition between the images of the unit of $A$, implies that
$\pi_\phi(a)=\pi_\psi(a)\chi_{\supp\phi(1)}$ for any $a\in A$.
\end{proof}

One can also introduce an order structure on the set $W(A,B)$, where $A$ and $B$
are any local \Cs-algebras, by setting $[\phi]\leq[\psi]$ whenever the two
c.p.c. order zero maps $\phi,\psi:A\to B$ are such that $\phi\precsim\psi$. Thus,
the bivariant Cuntz semigroup $(W(A,B),+,\leq)$ equipped with this order
relation becomes an \emph{ordered} Abelian monoid. With the following result we
justify the use of the word ``\emph{semigroup}'' in Definition \ref{def:bcu}.

\begin{proposition} Let $A$ and $B$ be any local \Cs-algebras. Then
$(W(A,B),+,\leq)$ is a positively ordered Abelian monoid.
\end{proposition}

\begin{proof} It is clear that $+$ is well-defined. It follows from
$\phi\hoplus\psi\sim\psi\hoplus\phi$ that $W(A,B)$ is Abelian. The class of the
zero map gives the neutral element with respect to $+$, and, moreover,
$0\precsim\phi$ for any c.p.c. order zero map $\phi:A\to
M_\infty(B)$, so $[0]\leq[\phi]$.
\end{proof}

Recall that every Abelian semigroup $S$ can be equipped with the algebraic
ordering relation, that is, $x\leq y$ in $S$ if there exists $z\in S$ for which
$x+z = y$. The order $\leq$ defined above extends the algebraic one. Indeed, if
$x,y\in W(A,B)$ are such that $x + z = y$ for some $z\in W(A,B)$, then any
representatives $\alpha,\beta,\gamma$ of $x,y,z$ respectively are obviously such
that $[\alpha]+[\gamma]=[\beta]$ by definition. This implies 
	$$\exists\seq b\subset M_\infty(B)\text{ such that }\norm{b_n^*\beta(a)b_n -
	(\alpha\hoplus\gamma)(a)}\to0\quad\forall a\in\A A.$$
Taking the sequence $b_n':=(u_n\otimes e_{11})b_n$, where $\seq u$ is an approximate
unit for $M_\infty(B)$, one has
	$$\norm{{b'_n}^*\beta(a)b'_n-\alpha(a)\otimes e_{11}}\to0$$
for any $a\in A$, whence $x\leq y$. However, like in the case of the ordinary
Cuntz semigroup, this order rarely agrees with the algebraic one. The following
example sheds some light on the relation between the bivariant Cuntz semigroup
just defined and the well-established Cuntz semigroup $W(\_)$.

\begin{example}\label{Ex:RecoverCuntz} Let $B$ be a \Cs-algebra, and let $\phi:\CC\to M_\infty(B)$ be a
c.p.c. order zero map. By \cite[Theorem 2.3]{wz2009}, there exists a positive
element $h\in M_\infty(B)^+$ such that
		$$\phi(z) = zh,\qquad\forall z\in\CC.$$
	Therefore, we can identify the set of c.p.c. order zero maps between $\CC$ and
	$M_\infty(B)$ with the positive cone of $M_\infty(B)$. If $\phi,\psi:\CC\to
	M_\infty(B)$ are c.p.c. order zero maps associated to the positive elements
	$h_\phi, h_\psi\in M_\infty(B)^+$ respectively, then $\phi\precsim\psi$
	in $W(\mathbb{C},B)$ if and only if $h_\phi\precsim h_\psi$ in $W(B)$. Hence,
	the map $\phi\mapsto h_\phi$ yields an isomorphism between $W(\CC,B)$ and
	$W(B)$.
\end{example}

The following technical result is a special instance of a more general result by
Handelman \cite{handelman} that applies to generic elements $s,t$ of a
\Cs-algebra that satisfy to $s^*s\leq t^*t$.

\begin{lemma}\label{lem:handelman} Let $A$ be a \Cs-algebra, and let $a,b\in
A^+$ be such that $a\leq b$. Then, there exists a sequence $\seq z\subset A$ of
contractions such that $z_nb^{\frac12}\to a^{\frac12}$.
\end{lemma}
\begin{proof} By functional calculus, one can check that the sought sequence
$\seq z$ is given by
	$$z_n = a^{\frac12}b^{\frac12}(b+\tfrac1n)^{-1},\qquad\forall n\in\IN,$$
and that each $z_n$ is such that $\norm{z_n}\leq 1$ for any $n\in\IN$.
\end{proof}

\begin{proposition}\label{prop:samesupp} Let $A$ and $B$ be local \Cs-algebras.
If $\phi,\psi:A\to B$ are two c.p.c. order zero maps with the same support
$*$-homomorphism $\pi:A\to\mathcal M(C)$, $C\subset B$, and such that $h_\phi\leq
h_\psi$ in $\mathcal M(C)$, then $\phi\precsim\psi$.
\end{proposition}
\begin{proof} Lemma \ref{lem:handelman} extends easily to local \Cs-algebras
because they are closed under functional calculus. Therefore, there exists a
sequence of contractions $\seq z$ in $\mathcal M(C)$ such that $z_nh_\psi^{\frac12}\to
h_\phi^{\frac12}$. By using an approximate unit of $C$, one may assume that
$\seq z\subset C$, with the property that $\norm{z_n\psi(a)z_n^*-\phi(a)}\to0$
for any $a\in A$. Hence, $\phi\precsim\psi$.
\end{proof}

As shown in \cite[Corollary 3.2]{wz2009}, one can perform continuous functional
calculus on any c.p.c. order zero map $\phi:A\to B$ between \Cs-algebras by
setting
	$$f(\phi)(a) = f(h_\phi)\pi_\phi(a),\qquad\forall a\in A,$$
where $h_\phi$ and $\pi_\phi$ are given by \cite[Theorem 2.3]{wz2009}, and $f$
is any function in $C_0((0,1])$. This result
generalizes to c.p.c. order zero maps between local \Cs-algebras.

\begin{proposition} Let $A,B$ be local \Cs-algebras, and let $\phi:A\to B$ be a
c.p.c. order zero map. For any function $f\in C_0((0,1])$, the map $f(\phi):A\to
B$ given by
		$$f(\phi)(a) := f(h)\pi(a)\qquad\forall a\in A,$$
	where $h$ and $\pi$ arise from Proposition \ref{prop:cpclocal}, is a c.p.c. order
	zero map between local \Cs-algebras.
\end{proposition}

\begin{proof} When $A$ is a \Cs-algebra the result follows from \cite{wz2009}
directly, so we assume that $A$ is not complete with respect to the topology of
the \Cs-norm. For any positive contraction $a\in A^+$, the \Cs-algebra
$A_a:=C^*(a)$ inside the completion of $A$ is $\sigma$-unital. Hence, the image
of the restriction of $\phi$ onto $A_a$ is contained inside a finitely generated
\Cs-subalgebra, say $B_\phi(a)$, of the completion of $B$, which is contained in
$B$ itself. To see that $B_\phi(a)$ is finitely generated, consider the
$*$-homomorphism $\rho_\phi:C_0((0,1])\otimes A\to B$ associated to $\phi$ according to \cite[Corollary
3.1]{wz2009}. Since $A_a$ is $\sigma$-unital, there exists a strictly positive
contraction $e\in A_a$. By setting $G_a:=\{\rho_\phi(t\otimes e),\phi(a)\}$, where
$t$ is the generator of $C_0((0,1])$, one sees that
	$$\phi(A_a) = \rho_\phi(t\otimes A_a) \subset \rho_\phi(C_0((0,1])\otimes A_a)=C^*(G_a),$$
so that we can take $B_\phi(a) := C^*(G_a)\subset B$.

Then, we will use the claim that for any polynomial $p$ of one variable and with zero constant term and degree $d>1$, $p(\phi)(a)$ belongs to $B_\phi(a)$. Indeed, for any such $p$, one can find another polynomial $P_p$ of zero constant term and of $m_d + 1$ variables, where $m_d = \lceil\log_2(d)\rceil$, such that
	$$p(\phi)(a) := p(h)\pi(a) =
		P_p(\phi(a),\phi(a^{\frac12}),\ldots,\phi(a^{\frac1{2^{m_d}}})) \in
		B_\phi(a).$$
	For example, if $p(x) = x^2$ then $p(\phi)(a) = h^2\pi(a) = h^2\pi(a^{\frac12})^2$, whence
		$$\phi^2(a) = \phi(a^\frac12)^2,$$
	so that $P_{x^2}(y_1, y_2) = y_2^2$. Similarly, if $p(x) = x^3$ then $p(\phi)(a) = h^3\pi(a) =	h\pi(a^{\frac12})h^2\pi(a^{\frac14})^2$, i.e.
		$$\phi^3(a) = \phi(a^{\frac12})\phi(a^{\frac14})^2.$$
	Hence, $P_{x^3}(y_1, y_2, y_3) = y_2y_3^2$. More generally one can verify that a suitable choice of polynomials $P_{x^k}$ is
		$$P_{x^k}(y_1,\ldots, y_{m_k+1}) = y_{m_k}^{2^{m_k}-k} y_{m_k+1}^{2k-2^{m_k}}$$
	for any $k\in\IN$, so that for
		$$p(x) = \sum_{k=1}^n a_kx^k,\qquad a_n\neq 0$$
	one has
		$$P_p(y_1,\ldots, y_{m_n+1}) = \sum_{k=1}^n a_ky_{m_k}^{2^{m_k}-k} y_{m_k+1}^{2k-2^{m_k}}.$$
 	Therefore, by approximating any function $f\in C_0((0,1])$ with polynomials $\{p_n\}$ having	zero constant term, one can set
		$$f(\phi)(a) := \lim_{n\to\infty}p_n(h)\phi(a)\in B_\phi(a),$$
	whose extension by linearity to $A$ defines the sought c.p.c. order zero map.
\end{proof}

\begin{corollary} Let $A$ and $B$ be local \Cs-algebras, $\phi:A\to B$ be a
c.p.c. order zero map and $f\in C_0((0,1])$ such that $x-f(x)\geq 0$ for all
$x\in(0,1]$. Then $f(\phi)\precsim\phi$.
\end{corollary}
\begin{proof} This result follows immediately from Proposition
\ref{prop:samesupp} since the maps $\phi$ and $f(\phi)$ share the same support
$*$-homomorphism, in the sense that $f(\phi)(a) = f(h_\phi)\pi_\phi(a)$ for any $a\in A$. Indeed, $f(\phi)=f(h)\pi_\phi$, and $f(h)\leq h$ in
$\mathcal M(C)$, where $C:=C^*(\phi(A))$.
\end{proof}

Let $f_\eps\in C_0((0,1])$ be the function defined by
	$$f_\eps(x) = \begin{cases}0 & x\in(0,\eps]\\ x-\eps &
	x\in(\eps,1],\end{cases}$$
that is, $f_\eps(x) = (x-\eps)_+$. For convenience we introduce the notation
$\phi_\eps$ by setting $\phi_\eps:=f_\eps(\phi)$ for any c.p.c. order zero map
$\phi:A\to B$ between the local \Cs-algebras $A$ and $B$.

\begin{corollary}\label{cor:cfcforcpc} Let $A$ and $B$ be local \Cs-algebras,
and let $\phi:A\to B$ be a c.p.c. order zero map. Then $\phi_\eps\precsim\phi$
for any $\eps > 0$.
\end{corollary}
\begin{proof} This follows from the fact that, for any contractive positive
element $h$ of a \Cs-algebra, one has $f_\eps(h)\leq h$ for any $\eps > 0$.
\end{proof}

\begin{lemma} Let $A$ and $B$ be local \Cs-algebras, and let $\phi:A\to B$ be a
c.p.c. order zero map. Then $\phi((a-\eps)_+)\geq(\phi(a)-\eps)_+$ for any $\eps
> 0$ and $a\in A^+$.
\end{lemma}
\begin{proof} By Proposition \ref{prop:cpclocal}, one has the decomposition
$\phi = h\pi$, where the positive element $h$ arises from the image of the unit
of the minimal unitization of $A$ through the unique c.p.c. order zero extension
$\phi^{(+)}:A^+\to B\dd$ of $\phi$. In particular, since $\norm\phi = \norm
h\leq 1$, one has
	\begin{align*}
		\phi^{(+)}(a-\eps 1_{A^+}) &= \phi(a) - \eps h\\
			&\geq\phi(a) - \eps1_{\mathcal{M}(C^*(\phi(A)))}.
	\end{align*}
	Considering the commutative C*-algebra generated by both sides of the above inequality, it follows that
	$(\phi(a)-\eps)_+\leq\phi^{(+)}(a-\eps1_{A^+})_+$. Moreover, from the
	positivity of $\phi$, one obtains the desired result, i.e.
	$\phi^{(+)}(a-\eps1_{A^+})_+ = \phi^{(+)}((a-\eps)_+) = \phi((a-\eps)_+)$.
\end{proof}

Observe that equality is attained when $\phi$ is a $*$-homomorphism rather than
a c.p.c. order zero map. We shall get back to this point in the next section,
where we introduce the notion of compact elements within the bivariant Cuntz
semigroup.

As a straightforward consequence of the already cited result of Handelman, i.e.
Lemma \ref{lem:handelman}, we have the following.
\begin{corollary}\label{cor:eps} Let $A$ and $B$ be local \Cs-algebras, and let
$\phi:A\to B$ be a c.p.c. order zero map. Then
$(\phi(a)-\eps)_+\precsim\phi((a-\eps)_+)$ for any $\eps > 0$ and $a\in A^+$.
\end{corollary}


\subsection{Additivity}

We now proceed to describe the behaviour of the bivariant Cuntz semigroup under
finite direct sums and show that it describes an additive bifunctor from the category of local \Cs-algebras to that of
partially ordered Abelian monoids.

Let $A_1$ and $A_2$ be local \Cs-algebras. Given two c.p.c. order zero maps
$\phi_1:A_1\to B$ and $\phi_2:A_2\to B$, their direct sum $\phi_1\oplus\phi_2$
is easily seen to be a c.p.c. order zero map. For the converse of this statement
we have the following.

\begin{lemma}\label{lem:cpcds} Let $A_1,A_2,B$ be local \Cs-algebras, and let
$\phi:A_1\oplus A_2\to B$ be a c.p.c. order zero map. Then, there are c.p.c.
order zero maps $\phi_1:A_1\to B$ and $\phi_2:A_2\to B$ such that
	\begin{enumerate}[i.]
		\item $\phi_1(A_1)\cap\phi_2(A_2) = \{0\}$;
		\item $\phi_1(a_1) + \phi_2(a_2) = \phi(a_1\oplus a_2)$.
	\end{enumerate}
	That is, $\phi = \phi_1\oplus\phi_2$.
\end{lemma}
\begin{proof} Define the maps $\phi_1:A_1\to B$ and $\phi_2:A_2\to B$ as
		$$\phi_1(a_1) := \phi(a_1\oplus
		0)\qquad\text{and}\qquad\phi_2(a_2):=\phi(0\oplus a_2)$$
	respectively. Clearly one has $\phi(a_1\oplus a_2) = \phi_1(a_1)+\phi_2(a_2)$.
	Suppose there exists an element $b\in\phi(A_1\oplus A_2)$ such that
	$b=\phi_1(a_1)=\phi_2(a_2)$ for some $a_1\in A_1$ and $a_2\in A_2$. By the
	\Cs-identity, we have
		\begin{align*}
			\norm b^2 &= \norm{b^*b}\\
				&= \norm{\phi_1(a_1)^*\phi_2(a_2)}\\
				&= \norm{\phi(a_1^*\oplus 0)\phi(0\oplus a_2)}\\
				&= 0,
		\end{align*}
	where we made use of the fact that $(a_1^*\oplus 0)\perp (0\oplus a_2)$ and
	that $\phi$ has the order zero property. Therefore, $\norm b=0$, whence $b=0$,
	i.e. $\phi_1(A_1)\cap\phi_2(A) = \{0\}$.
\end{proof}

\begin{theorem} For any triple of local \Csalg s $A_1$, $A_2$ and $B$, the
semigroup isomorphism
		$$W(A_1\oplus A_2, B)\cong W(A_1,B)\oplus W(A_2, B)$$
	holds.
\end{theorem}
\begin{proof} Let $\sigma : W(A_1,B)\oplus W(A_2, B)\to W(A_1\oplus A_2,B)$ be
the map given by
		$$\sigma([\phi_1]\oplus[\phi_2]) = [\phi_1\oplus\phi_2].$$
	By Lemma \ref{lem:cpcds} the map $\sigma$ is surjective. To prove injectivity we show that $\phi_1\oplus\phi_2\precsim\psi_1\oplus\psi_2$
	implies $\phi_k\precsim\psi_k$, $k=1,2$. By hypothesis, there exists a
	sequence $\seq b\subset M_\infty(B)$ such that
	$\norm{b_n^*(\psi_1(a_1)\oplus\psi_2(a_2))b_n-\phi_1(a_1)\oplus\phi_2(a_2)}
	\to 0$ for every $a_1\in A_1,a_2\in A_2$. Considering $M_2(M_\infty(B))\cong
	M_\infty(B)$, the sequence $b_n$ has the structure
		$$b_n = \sum_{i,j=1}^2b_{n,ij}\otimes e_{ij},$$
	where $b_{n,ij}\in M_\infty(B)$ for any $i,j=1,2$, and $\{e_{ij}\}_{i,j=1,2}$
	form the standard basis of matrix units of $M_2$. Thus, for $a_2 = 0$, one
	finds that $b_{n,11}^*\psi_1(a_1)b_{n,11}\to \phi_1(a_1)$ in norm for any
	$a_1\in A_1$, i.e. $\phi_1\precsim\psi_1$. A similar argument with $a_1=0$
	leads to $\phi_2\precsim\psi_2$.
	
	To check that $\sigma$ preserves the semigroup operations, it suffices to show
	that
	$(\phi_1\hoplus\psi_1)\oplus(\phi_2\hoplus\psi_2) \sim
	(\phi_1\oplus\phi_2)\hoplus(\psi_1\oplus\psi_2)$.
	A direct computation reveals that such equivalence is witnessed by the
	sequence $\seq c\subset M_4(M_\infty(B))$ given by $c_n:= u_n\otimes(e_{11} +
	e_{44} + e_{23} + e_{32})$, where $\seq u\subset M_\infty(B)$ is an
	approximate unit.
\end{proof}

We observe that countable additivity in the first argument can be obtained when
the algebra in the second argument is stable. This is because the countable
direct sum of maps with codomain in $M_\infty(B)$ could take values in $B\otimes
K$ instead, by the definition of the direct sum. In this case, an isomorphism
between $W(\bigoplus_{n\in\IN}A_n,B)$ and $\bigoplus_{n\in\IN}W(A_n,B)$ is
provided by the map
	$$\sigma\left(\bigoplus_{n\in\IN}[\phi_n]\right):=
	\left[\bigoplus_{n\in\IN}\frac1{2^n}\phi_n\right].$$

\begin{theorem} For any triple of local \Cs-algebras $A$, $B_1$ and $B_2$, the
semigroup isomorphism
		$$W(A, B_1\oplus B_2)\cong W(A,B_1)\oplus W(A, B_2)$$
	holds.
\end{theorem}
\begin{proof} Since $M_\infty(B_1\oplus B_2)$ is isomorphic to
$M_\infty(B_1)\oplus M_\infty(B_2)$, one has that for every c.p.c. order zero
map $\phi:A\to M_\infty(B_1\oplus B_2)$ there are c.p.c. order zero maps
$\phi_k:A\to M_\infty(B_k)$, $k=1,2$ such that $\phi$ can be identified, up to
isomorphism, with $\phi_1\hoplus\phi_2$\footnote{such maps are given by $\phi_k
:= \pi_k^{(\infty)}\circ\phi$, where $\pi_k^{(\infty)}$ is the
$\infty$-ampliation of the natural projection $\pi_k:B_1\oplus B_2\to B_k$, for $k=1,2$, that
is, $\pi_k^{(\infty)}:= \pi_k\otimes\id_{M_\infty}$.}. This shows that the map
$\rho: W(A,B_1)\oplus W(A,B_2)\to W(A,B_1\oplus B_2)$ given by
		$$\rho([\phi_1]\oplus[\phi_2]) := [\phi_1\hoplus\phi_2]$$
is surjective. Injectivity comes from the fact that
$\phi_1\hoplus\phi_2\precsim\psi_1\hoplus\psi_2$ implies $\phi_1\precsim\psi_1$
and $\phi_2\precsim\psi_2$.
\end{proof}


\subsection{Functoriality} We proceed by showing that $W(\_,\_)$ can be viewed as a
functor from the bicategory ${\cat{C}^*\loc}\op\times\cat{C}^*\loc$ to
$\cat{OrdAMon}$, where $\cat{C}^*\loc$ denotes the category of local
\Cs-algebras, and $\cat{OrdAMon}$ that of ordered Abelian monoids. Further
categorical aspects are confined to Section \ref{sec:cats}, where it is shown
that the target category can be enriched with extra structure.

\begin{proposition} Let $B$ be a local \Cs-algebra. Then, $W(\ \cdot\ ,B)$ is a
contravariant functor from the category of local \Cs-algebras to that of ordered
Abelian monoids.
\end{proposition}
\begin{proof} Let $A,A'$ be arbitrary local \Cs-algebras. Consider a
$*$-homomorphism $f\in\Hom(A,A')$ and a c.p.c. order zero map
$\psi':A'\to M_\infty(B)$. Define $f^*(\psi')$ in such a way that the
diagram
	$$\xymatrix@R=0.5em@C=3em{%
		A  \ar[dr]^-{f^*(\psi')}\ar[dd]_{f} &             \\
		                                    & M_\infty(B) \\
		A' \ar[ur]_-{\psi'}                 & 
	}$$
commutes, i.e. $f^*(\psi') := \psi'\circ f$. Then $f^*(\psi)$ is a c.p.c. 
order zero between $A$ and $M_\infty(B)$. Therefore, $f^*$
defines a pull-back between c.p.c. order zero maps which can be projected onto
equivalence classes from the corresponding bivariant Cuntz semigroups by setting
	$$W(f,B)([\psi']) = [f^*(\psi')],\qquad\text{ for all } [\psi']\in W(A',\A
	B).$$
This yields a well-defined map. It implies that for every $*$-homomorphism
$f$ there exists a semigroup homomorphism $W(f,B)$ such that the following
diagram
	$$\xymatrix{%
		A           \ar[r]^f\ar[d]_{W(\ \cdot\ ,B)}  &
		A'          \ar[d]^{W(\ \cdot\ ,B)}            \\
		W(A,B)                                          &
		W(A', B) \ar[l]_{W(f,B)}
	}$$
commutes. To see that such map preserves the order consider another c.p.c. order
zero map $\phi':A'\to M_\infty(B)$ with $\phi'\precsim\psi'$. Then there exists
a sequence $\seq b\subset M_\infty(B)$ such that
$\norm{b_n^*\psi'(a)b_n-\phi'(a)}\to0$ for any $a\in A'$. In particular this is
true if $a$ is restricted to $f(A)\subset A'$, whence $f^*(\phi)\precsim f^*(\psi)$.
\end{proof}

\begin{proposition} Let $A$ be a local \Cs-algebra. Then, $W(A,\ \cdot\ )$ is a
covariant functor from the category of local \Cs-algebras to that of ordered
Abelian monoids.
\end{proposition}
\begin{proof} Let $B$ and $B'$ be arbitrary local \Cs-algebras. Take a
$*$-homomorphism $g\in\Hom(B, B')$ and a
c.p.c. order zero $\psi$ between $A$ and $M_\infty(\A
B)$. Define $g_*(\psi)$ such that the diagram
	$$\xymatrix@R=0.5em@C=3em{%
		     & M_\infty(B)\ar[dd]^{g^{(\infty)}} \\
		A \ar[ru]^-{\psi}\ar[rd]_-{g_*(\psi)} & \\
		     & M_\infty(B')
	}$$
commutes, i.e. define
	$$g_*(\psi) := g^{(\infty)}\circ\psi.$$
Such map is clearly completely positive with the order zero property and
well-defined; therefore, the above line defines a push-forward between c.p.c.
order zero maps that gives rise to the semigroup homomorphism
	$$W(A, g)([\psi]) = [g_*(\psi)],\qquad\text{ for all}[\psi]\in W(A,B).$$
If $W(A,\ \cdot\ )$ denotes the functor $B\mapsto W(A, B)$, where $B$
is any local \Cs-algebra, then the above definition implies that the diagram
	$$\xymatrix{%
		B           \ar[r]^g\ar[d]_{W(A,\ \cdot\ )}  &
		B'          \ar[d]^{W(A,\ \cdot\ )}            \\
		W(A,B)   \ar[r]^{W(A, g)}                 &
		W(A, B')
	}$$
commutes. To see that such map preserves the order consider another c.p.c. order
zero map $\phi:A\to M_\infty(B)$ such that $\phi\precsim\psi$. Then there exists
a sequence $\seq b\subset M_\infty(B)$ such that
$\norm{b_n^*\psi(a)b_n-\phi(a)}\to0$ for any $a\in A'$. Since $g$ is necessarily
contractive, the sequence $\{g^{(\infty)}(b_n)\}_{n\in\IN}\subset M_\infty(B')$
witnesses the relation $g_*(\phi)\precsim g_*(\psi)$.
\end{proof}


\subsection{Stability} It follows directly from the definition of the bivariant
Cuntz semigroup that matrix stability holds trivially on the second argument,
namely
	$$W(A,M_n(B))\cong W(A,B),\qquad\forall n\in\IN,$$
for any pair of local \Cs-algebras $A$ and $B$. This is just a special instance
of the more general stability property
	$$W(A,\Mi(B))\cong W(A,B),$$
which also follows from the definition of the bifunctor $W(\_,\_)$. On the other hand,
matrix stability on the first argument is not as immediate. In order to
establish this property, we first record some technical results.

\begin{proposition}\label{prop:compositions} Let $A$, $B$ and $C$ be local
\Cs-algebras, and let $\phi,\psi:A\to B$, $\eta,\theta:B\to C$ be c.p.c. order
zero maps such that $\phi\precsim\psi$ and $\eta\precsim\theta$. Then
$\eta\circ\phi\precsim\eta\circ\psi$ and $\eta\circ\phi\precsim\theta\circ\phi$.
\end{proposition}
\begin{proof} The second implication is trivial. To prove the first one, let $\seq e\subset C^*(\eta(B))\cap C$ be an
approximate unit, and let $\pi_\eta$ be the support $*$-homomorphism of $\eta$.
If $\seq b\subset B$ is any sequence that witnesses $\phi\precsim\psi$, then the
sequence $\seq{d}\subset C^*(\eta(B))$ given by $d_n:=e_n\pi_\eta(b_n)$ for any
$n\in\IN$ can be perturbed into a sequence $\seq c$ into the dense subalgebra
$C^*(\eta(B))\cap C$ with the property that $\norm{d_n-c_n}<\frac1n$. Hence
	\begin{align*}
		\norm{c_n(\eta\circ\psi)(a)c_n^*-(\eta\circ\phi)(a)}
			&=\norm{c_n(\eta\circ\psi)(a)c_n^*-d_n(\eta\circ\psi)(a)d_n^*}\\
			&\qquad+\norm{e_n\eta(b_n\psi(a)b_n^*)e_n^*-(\eta\circ\phi)(a)}\\
			&\leq
			\tfrac2n+\norm{e_n\eta(b_n\psi(a)b_n^*)e_n^*-e_n(\eta\circ\phi)(a)e_n^*}\\
			&\qquad+\norm{e_n(\eta\circ\phi)(a)e_n^* - (\eta\circ\phi)(a)}\\
			&\leq\tfrac 2n+\norm{b_n\psi(a)b_n^*-\phi(a)}\\
			&\qquad+\norm{e_n(\eta(\phi(a)))e_n^*-\eta(\phi(a))}.
	\end{align*}
	Since this last term tends to 0 as $n\to\infty$, for every $a\in A$,
	we have that $\eta\circ\phi\precsim\eta\circ\psi$.
\end{proof}

\begin{lemma}\label{lem:tensorcse} Let $A$, $B$, $C$ and $D$ be local
\Cs-algebras, and let $\phi,\psi:A\to B$, $\eta:C\to D$ be c.p.c. order zero
maps such that $\phi\precsim\psi$. Then
$\eta\otimes\phi\precsim\eta\otimes\psi$.
\end{lemma}
\begin{proof} If $\seq b\subset B$ is the sequence that witnesses the Cuntz
subequivalence between $\phi$ and $\psi$, then $\{e_n\otimes b_n\}_{n\in\IN}$,
where $\seq e\subset D$ is an approximate unit, witnesses the sought Cuntz
subequivalence between $\eta\otimes\phi$ and $\eta\otimes\psi$.
\end{proof}

\begin{corollary}\label{cor:stabilitycse} Let $A$ and $B$ be local \Cs-algebras,
and let $\phi,\psi:A\to B$ be c.p.c. order zero maps. One has $\phi\precsim\psi$
in $B$ if and only if $\phi\otimes\id_\Mi\precsim\psi\otimes\id_\Mi$ in
$B\otimes\Mi$. The same holds true with $K$ in place of $\Mi$.
\end{corollary}
\begin{proof} The fact that
$\phi\precsim\psi$ implies $\phi\otimes\id_\Mi\precsim\psi\otimes\id_\Mi$
follows from the previous lemma. For the other implication observe that $B$
embeds into $B\otimes\Mi$ by means of the injective map $b\stackrel\iota\mapsto
b\otimes e$, where $e\in\Mi$ is any minimal projection. Fix $\seq b\subset
B\otimes\id_\Mi$ being the sequence that witnesses the relation
$\phi\otimes\id_\Mi\precsim\psi\otimes\id_\Mi$. Then, with $x_n:=(1_{B^+}\otimes
e)b_n(1_{B^+}\otimes e)\in B\otimes \{e\}$, where $1_{B^+}$ is either the unit
of $B$, or that of its minimal unitization $B^+$ in the non-unital case, we have
	$$\norm{x_n^*(\psi(a)\otimes e)x_n-\phi(a)\otimes e}\to 0,\qquad\forall a\in\
	A.$$
It can be pulled back to $B$ through $\iota$ giving
	$$\norm{\iota^{-1}(x_n)^*\psi(a)\iota^{-1}(x_n)-\phi(a)}\to 0,\qquad\forall
	a\in A,$$
whence $\phi\precsim\psi$. The same argument works with $K$ in place of $\Mi$.
\end{proof}

Recall that for any isomorphism $\gamma:M_\infty\otimes M_\infty \to M_\infty$ there is an isometry $v\in B(\ell^2(\IN))$ such that $\Ad_v\circ\gamma\circ(\id_{M_\infty}\otimes e) = \id_{M_\infty}$, where $e\in M_\infty$ is any minimal projection. The same holds true with $K$ in place of $M_\infty$.

\begin{proposition}\label{prop:stability} Let $A$ and $B$ be local \Cs-algebras.
Then $W(M_\infty(A), B)\cong W(A,B)$ and $W(A\otimes K,B\otimes K)\cong W(A,
B\otimes K)$.
\end{proposition}
\begin{proof} Since the isomorphism $\gamma:M_\infty\otimes M_\infty\to
M_\infty$ induces a semigroup isomorphism $W(A,M_\infty(B))\cong W(A,B)$, it is
enough to show that one has $W(A,B)\cong W(M_\infty(A),M_\infty(B))$. To this end, we will use the fact that if $e$ is a minimal projection in $M_\infty$, then there exists an isometry $v\in B(H)$ such that $v^*\gamma\circ(id_{M_\infty}\otimes e)v=id_{M_\infty}$. Namely, they are Cuntz equivalent.

In this case, mutual
inverses are then given by the maps
	$$[\phi]\mapsto[\phi\otimes\id_{M_\infty}],\qquad[\phi]\in W(A,B)$$
and
	$$[\Phi]\mapsto[(\id_B\otimes\gamma)\circ\Phi\circ(\id_A\otimes e)], \qquad\Phi\in W(M_\infty(A),M_\infty(B)).$$

Indeed, by making use of
Proposition \ref{prop:compositions} and Lemma \ref{lem:tensorcse} above, one has
	\begin{align*}
		(\id_B\otimes\gamma)\circ(\phi\otimes\id_{M_\infty})\circ(\id_A\otimes e) &=
		(\id_B\otimes\gamma)\circ(\phi\otimes e)\\
			&= (\id_B\otimes\gamma)\circ(\id_B\otimes\id_{M_\infty}\otimes
			e)\circ\phi\\
			&\sim_\Cu (\id_B\otimes\id_{M_\infty})\circ\phi\\
			& =\phi
	\end{align*}
	and
	\begin{align*}
		((\id_B\otimes\gamma)\circ\Phi\circ(\id_A\otimes e))&\otimes\id_{M_\infty}=\\
			&=
			(\id_B\otimes\gamma\otimes\id_{M_\infty}) \circ
			(\Phi\otimes\id_{M_\infty})\circ(\id_A\otimes
			e\otimes\id_{M_\infty})\\
			&\sim_\Cu
			(\id_B\otimes\gamma\otimes\id_{M_\infty}) \circ
			(\Phi\otimes\id_{M_\infty})\circ(\id_A\otimes\id_{M_\infty}\otimes
			e)\\
			&= (\id_B\otimes\gamma\otimes\id_{M_\infty})\circ(\Phi\otimes e)\\
			&= (\id_B\otimes\gamma\otimes\id_{M_\infty}) \circ
			(\id_B\otimes\id_{M_\infty}\otimes\id_{M_\infty}\otimes e)\circ\Phi\\
			&\sim_\Cu (\id_B\otimes\id_{M_\infty}\otimes\id_{M_\infty})\circ(\id_B\otimes\id_{M_\infty}\otimes e\otimes \id_{M_\infty})\circ\Phi\\
			&\sim_\Cu(\id_B\otimes\id_{M_\infty}\otimes \id_{M_\infty})\circ\Phi\\
			&= \Phi,
	\end{align*}
	which become equalities at the level of the Cuntz classes. The result
	involving $K$ follows in a similar manner.
\end{proof}

\begin{remark}\label{rem:Repre} As a corollary of the above result, we have that every c.p.c. order zero map
$\Phi:A\otimes K\to B\otimes K$ is Cuntz-equivalent to a $K$-ampliation of a
c.p.c. order zero map $\phi:A\to B\otimes K$, that is, $\Phi \sim
\phi\otimes\id_K$, up to the identification $K\otimes K\cong K$.
\end{remark}

We conclude this section by observing that one does not have the semigroup isomorphism $W(A,B\otimes K)\cong
W(A,B)$ in general, unless $B$ is a stable \Cs-algebra. An easy counterexample
is the case $A,B = \CC$, where $W(\CC,\CC\otimes K)\cong
W(K)\cong\IN_0\cup\{\infty\}\neq \IN_0\cong W(\CC) \cong W(\CC,\CC)$.


\subsection{The bivariant functor $WW(\_,\_)$} Switching our attention to \Cs-algebras
rather than local \Cs-algebras, we now introduce a new bifunctor, denoted by $WW(\_,\_)$, over the
bicategory ${\cat C^*}\op\times \cat C^*$, by setting
	$$WW(A,B) := W(A\otimes K,B\otimes K).$$
It can be considered as the analogue of the stabilization, $\Cu(\_)$, of the
ordinary functor $W(\_)$. Note that through Proposition \ref{prop:stability}
that stabilization is only needed in the second argument, so that we actually
have the equivalent characterization
	$$WW(A,B) \cong W(A, B\otimes K).$$
We formalize these considerations as follows.

\begin{definition} Let $A$ and $B$ be \Cs-algebras. The bivariant Cuntz
semigroup $WW(A,B)$ is the set of equivalence classes
		$$WW(A,B) = \{\phi:A\otimes K\to B\otimes K\ |\ \phi\text{ is c.p.c. order
		zero}\}/\sim,$$
	endowed with the binary operation given by the direct sum $\hoplus$.
\end{definition}

In the above definition, the sequences witnessing (sub)equivalence are of course
required to be in $B\otimes K$ rather than in $M_\infty(B)$. It is also easy to
check that all the properties of the bifunctor $W(\_,\_)$ that have been
observed extend naturally to this new bifunctor $WW(\_,\_)$. However, an
advantage of such description is that matrix stability on both arguments holds
trivially as a consequence of the more general stability property; namely,
	$$WW(A\otimes K,B\otimes K)\cong WW(A,B),$$
 holds for any pair of \Cs-algebras $A$ and $B$.


\subsection{The Module Picture}

Similarly to the ordinary Cuntz semigroup and KK-theory, the bivariant Cuntz
semigroup $WW(A,B)$ of the previous section can be formulated in terms of Hilbert
modules.

\begin{definition}[Order zero pair] Let $A$ and $B$ be \Cs-algebras. An $A$-$B$
order zero pair is a pair $(X,\phi)$ consisting of a countably generated Hilbert
$B$-module $X$ and a non-degenerate c.p.c. order zero map $\phi:A\to K(X)$.
\end{definition}

If $(X,\phi)$ and $(Y,\psi)$ are $A$-$B$ order zero pairs, we say that
$(X,\phi)$ is Cuntz subequivalent to $(Y,\psi)$, $(X,\phi)\precsim(Y,\psi)$ in
symbols, if there exists a sequence $\seq s\in K(X,Y)$ such that
	$$\lim_{n\to\infty}\norm{s_n^*\psi(a)s_n-\phi(a)} = 0,$$
for all $a\in A$. The antisymmetrization of such a subequivalence relation gives
an equivalence relation, namely $(X,\phi)\sim(Y,\psi)$ if
$(X,\phi)\precsim(Y,\psi)$ and $(Y,\psi)\precsim(X,\phi)$.
\begin{definition}
For separable \Cs-algebras $A$ and $B$, we define
		$$\mathcal WW(A,B):=\{A\otimes K-B\otimes K\text{ order zero pairs}\}/\sim,$$
	endowed with the binary operation arising from the direct sum of pairs, i.e.
		$$[(X,\phi)]+[(Y,\psi)] := [(X\oplus Y,\phi\hoplus\psi)].$$
\end{definition}

The above definition can be taken to provide a module picture for the bivariant
Cuntz semigroup $WW(A,B)$, as it is described in the following result.

\begin{theorem} For any pair of separable \Cs-algebras $A$ and $B$, there is a
natural isomorphism $$WW(A,B)\cong\mathcal WW(A,B).$$
\end{theorem}
\begin{proof} By Kasparov's stabilization theorem, one has the identifications
		$$K(X)\subset K(X\oplus H_B)\cong K(H_B)\cong B\otimes K.$$
	Observe that the map that sends an $A\otimes K$-$B\otimes K$ order zero pair
	$(X,\phi)$ to the c.p.c. order zero map
		$$\phi : A\otimes K\to K(X)\subset B\otimes K$$
	has the map that sends a c.p.c. order zero map $\phi:A\otimes K\to B\otimes K$
	to the pair
		$$(\overline{\phi(A\otimes K)H_B},\phi)$$
	as an inverse. In particular, both maps preserve Cuntz subequivalence.
	Indeed, let $(X,\phi)\precsim(Y,\psi)$, i.e. there exists a sequence $\seq s$ in $K(X,Y)$ such
	that $\norm{s_n^*\psi(a)s_n-\phi(a)}\to 0$ as $n\to\infty$ for all $a\in
	A\otimes K$. Concretely, $\seq s\subset K(X,Y)\subset K(X\oplus H_B,Y\oplus
	H_B)\subset K(H_B)\cong B\otimes K$; hence, up to this identification,
	$\norm{s_n^*\psi(a)s_n-\phi(a)}\to 0$ i.e. $\phi\precsim\psi$. Conversely,
	let $\phi\precsim\psi$, so that there exists $\seq z\subset B\otimes K\cong
	K(H_B)$ such that $\norm{z_n^*\phi(a)z_n-\psi(a)}\to 0$ for all $a\in
	A\otimes K$. Since $\overline{\phi(A\otimes K)H_B}$ and
	$\overline{\psi(A\otimes K)H_B}$ are countably generated Hilbert modules,
	there are projections $p,q\in B(H_B)$ such that $pH_B=\overline{\phi(A\otimes K)H_B}$
	and $qH_B=\overline{\psi(A\otimes K)H_B}$. Thus, the sequence $\seq w\subset
	K(\overline{\phi(A\otimes K)H_B},\overline{\psi(A\otimes K)H_B})$ given by
		$$w_n := pz_nq\in K(\overline{\phi(A\otimes K)H_B},\overline{\psi(A\otimes
		K)H_B})$$
	is such that $\norm{w_n^*\psi(a)w_n-\phi(a)}\to0$. This shows precisely that
	$(\overline{\phi(A\otimes K)H_B},\phi)$ is subequivalent to $(\overline{\psi(A\otimes
	K)H_B},\psi)$.
\end{proof}
	

\section{\label{sec:cats}Further Categorical Aspects}

One of the main drawback of the standard Cuntz semigroup is the lack of continuity under arbitrary inductive limits of the functor $W(\_)$ as an invariant for C*-algebras. This problem was remedied in \cite{cei2008} defining a richer category of semigroups, called $\Cu$, and considering the stabilized Cuntz semigroup, called $\Cu(\_)$, instead of the usual one. In particular, $\Cu(A)\cong W(A\otimes K)$ for any C*-algebra $A$. 

Further on, it is shown in \cite{apt2014} that, with a suitable choice of the source and target categories, the functor $W(\_)$ exhibits the \emph{good} functoriality properties required before. Indeed, if
one takes the category of \emph{local} \Cs-algebras  instead of the C*-algebras as the source, and the
category $\cat W$, introduced in \cite{apt2014}, as the target category, one
obtains such properties. In light of this result, it is natural to ask whether
the bivariant Cuntz semigroup functor $W(\_,\_)$ defined in the previous section
belongs to the category $\cat W$, and what properties it possesses if one can choose
this enriched category as the target for $W(\_,\_)$.

We start by recalling the definition of the category $\cat W$. To this end, and
following \cite{apt2014}, we introduce the notion of an \emph{auxiliary
relation} on a partially ordered Abelian monoid.

\begin{definition}[Auxiliary relation] Let $(S,\leq)$ be a partially ordered
monoid. An auxiliary relation on $S$ is a binary relation $\prec$ such that for $a,b,c,d \in S$:
	\begin{enumerate}[(i)]
		\item $a\prec b$ implies $a\leq b$.
		\item $a\leq b\prec c\leq d$ implies $a\prec d$.
		\item $0\prec a$.
	\end{enumerate}
\end{definition}

Let $(S,\leq,\prec)$ be a partially ordered Abelian monoid equipped with an
auxiliary relation $\prec$. Given an element $a\in S$, we adopt the notation
$a^\prec\subset S$ for the subset of $S$ generated by $a$ and the relation
$\prec$ as
	$$a^\prec:=\{x\in S\mid x\prec a\}.$$

\begin{definition}[Category $\cat{W}$]\label{W} Let $\cat{W}$ be the category
whose objects are partially ordered monoids $(S, \leq,\prec)$, equipped with an
auxiliary relation $\prec$, that satisfy the following properties.
	\begin{enumerate}[{\normalfont(WO.1)}]
		\item For any $a\in S$, $a^\prec$ is upward directed and it contains a
		cofinal $\prec$-increasing sequence (i.e., there is a sequence $a_1\prec
		a_2\prec\ldots \in a^\prec$ such that for each $b\prec a$, there is $k$ such
		that $b\leq a_k$);
		\item for any $a\in S$, $a^\prec$ admits a supremum and $\sup a^\prec = a$;
		\item if $a'\prec a$ and $b'\prec b$, then $a'+b'\prec a+b$ (i.e. $a^\prec
		+b^\prec\subset (a+b)^\prec$);
		\item $a^\prec+b^\prec$ is cofinal in $(a+b)^\prec$.
	\end{enumerate}
The morphisms of $\cat{W}$ are semigroup homomorphisms $\Phi:S\to T$ that
satisfy the following axioms.
	\begin{enumerate}[{\normalfont(WM.1)}]
		\item Continuity, i.e. for every $s\in S$ and $t\in T$ with $t\prec \Phi(s)$
		there exists $s'\in S$ such that $s'\prec s$ and $t\leq\Phi(s')$;
		\item $\Phi$ preserves the auxiliary relation $\prec$, i.e.
		$\Phi(a)\prec\Phi(b)$ in $T$ whenever $a\prec b$ in $S$.
	\end{enumerate}
\end{definition}

In \cite{apt2014}, a partially ordered monoid homomorphism that satisfies the
property (WM.1) alone is called a \emph{generalized $\cat{W}$-morphisms}.

Let $A$ be a local \Cs-algebra and equip its Cuntz semigroup $W(A)$ with the
relation $\prec$ given by
	\begin{equation}\label{eq:ar}
		[a]\prec[b]\,\,\text{ if and only if there exists }\,\,\eps> 0\ \text{ such
		that }\ [a]\leq[(b-\eps)_+].
	\end{equation}
It is easy to see that $\prec$ defines an auxiliary relation on the partially
ordered Abelian monoid $W(A)$, and that $(W(A),\prec)$ belongs to the category
$\cat W$ defined above (cf. \cite[Proposition 2.2.5]{apt2014}). Recall that the so-called {\it way-below} relation ($a\ll b$ if whenever $b\leq \sup y_n$, then there exists an $n\in\mathbb{N}$ such that $a\leq y_n$)
is the natural auxiliary relation considered in $W(A)$.

\begin{proposition}\label{prop:MorphismsW} Let $A$ and $B$ be local
\Cs-algebras. Every c.p.c. order zero map $\phi:A\to B$ naturally induces a
generalized $\cat{W}$-morphism $W(\phi):W(A)\to W(B)$. If $\phi$ is a
$*$-homomorphism, then $W(\phi)$ preserves the auxiliary relation and thus is a
$\cat{W}$-morphism.
\end{proposition}
\begin{proof} It follows from \cite[Corollary 4.5]{wz2009} that $W(\phi)$ is a
well-defined morphism between semigroups (i.e. $W(\phi)$ preserves addition,
order and the zero element).

To check that $W(\phi)$ is continuous, let $t\in W(B)$ and $s\in W(A)$ be such
that $t\prec W(\phi)(s)$. We need to show the existence of $s'\in W(A)$ such
that $s'\prec s$ and $t\leq W(\phi(s'))$. To this end, let $[x]=s$. Since
$t\prec W(\phi)(s)=[\phi(x)]$, there exists $\e>0$ such that
$t\leq[(\phi(x)-\e)_+]$. Moreover, from Corollary \ref{cor:eps} we have that
$(\phi(x)-\e)_+\precsim\phi((x-\e)_+)$. Therefore, by setting
$s':=[(x-\e)_+]$, we have $s'\prec s$ in $W(A)$ and $t\leq W(\phi)(s')$.
\end{proof}

\begin{remark}
Observe that, if $\phi$ and $\psi$ are Cuntz equivalent c.p.c. order zero maps,
then the induced maps at the level of the Cuntz semigroups are the same. Indeed, the sequence that witnesses the Cuntz equivalence between $\phi$ and $\psi$ can be used to show that $W(\phi)(a)=W(\psi)(a)$ for all $a\in A$. 
\end{remark}
We now turn our attention to the question of whether the bivariant Cuntz
semigroup $W(A,B)$ belongs to the category $\cat W$ mentioned above. First of
all we introduce the following auxiliary relation on $W(A,B)$, as a
generalization of the above introduced auxiliary relation \eqref{eq:ar}.

\begin{definition} Let $A$ and $B$ be local \Cs-algebras, and let $\Phi,\Psi\in
W(A,B)$. Define the auxiliary relation $\prec$ on $W(A,B)$, in symbols
$\Phi\prec \Psi$, if there exists $\e>0$ such that $\Phi\leq [\psi_\e]$, where
$\psi$ is any representative of $\Psi$.
\end{definition}
It is left to the reader to check that the above definition indeed gives an
auxiliary relation.

\begin{lemma} Let $A$ and $B$ be local \Cs-algebras, and let $\Phi\in W(A,B)$.
For any representative $\phi\in\Phi$, the sequence $\seq\Phi$, given by $\Phi_n
:= [\phi_{\frac1n}]$, is increasing in $W(A,B)$ and is such that $\sup\Phi_n =
\Phi$.
\end{lemma}
\begin{proof} Given $\eps_1,\eps_2 > 0$, one has $\phi_{\eps_1+\eps_2} =
(\phi_{\eps_1})_{\eps_2}$. Hence, the sequence $\Phi_n$ is increasing by
Corollary \ref{cor:cfcforcpc}. Moreover, from the same corollary, we have that
$\Phi_n\leq\Phi$ for any $n\in\IN$, whence $\sup\Phi_n\leq\Phi$. Suppose
$\Psi\in W(A,B)$ is such that $\Phi_n\leq\Psi$ for any $n\in\IN$, and let $\psi$
be any representative of $\Psi$. From the local description of Cuntz comparison
of c.p.c. order zero maps (Proposition \ref{prop:subeq}), we have that for all $n\in\IN$, any finite subset $F$
of $A$ and given $\eps>0$, there exists
		$$ b_{n,F,\eps}\in M_\infty(B) \;\text{ such that
		}\;\norm{b_{n,F,\eps}^*\psi(a)b_{n,F,\eps}-\phi_{\frac1n}(a)}<\e,\;\text{
		for all } a\in A.$$
	Since the continuous functional calculus is norm-continuous, i.e.
		$$\norm{\phi_{\frac1n}(a)-\phi(a)}\to0,\qquad \forall a\in A,$$
	it follows that for each $a\in A$ and $\eps>0$, there exist $n_{a,\eps}\in\IN$ such that  $n>n_{a,\eps}$ implies
		$$\norm{\phi_{\frac1n}(a)-\phi(a)} < \e.$$
	Moreover, for any finite subset $F\Subset A$ and $\eps>0$, one can take
	$N_\eps := \max_{a\in F} \{n_{a,\eps}\}+1$, so that there exists
	$b_{N_\eps,F,\eps}\in M_\infty(B)$ with the property that
		$$\norm{b_{N_\eps,F,\eps}^*\psi(a)b_{N_\eps,F,\eps} -
		\phi_{\frac1{N_\eps}}(a)}<\eps,\text{for all } a\in F.$$
	Setting $b := b_{N_\eps,F,\eps}^*$, one has that
		\begin{align*}
			\norm{b^*\psi(a)b-\phi(a)} &= \norm{b^*\psi(a)b_n - \phi_{\frac1{N_\e}}(a)
			+ \phi_{\frac1{N_\e}}(a) - \phi(a)}\\
				&\leq\norm{b^*\psi(a)b_n - \phi_{\frac1{N_\e}}(a)} +
				\norm{\phi_{\frac1{N_\e}}(a) - \phi(a)}\\
				&< 2\eps
		\end{align*}
	for any $a\in F$. Therefore, $\Phi\leq\Psi$. By the arbitrariness of $\Psi$,
	we conclude that $\Phi = \sup\Phi_n$.
\end{proof}

\begin{proposition} Let $A$ and $B$ be local \Cs-algebras. The bivariant Cuntz
semigroup $W(A,B)$ is an object of the category $\cat{W}$.
\end{proposition}
\begin{proof} One has to verify that $W(A,B)$ has all the properties of
Definition \ref{W}. By the definition of the auxiliary relation $\prec$ on
$W(A,B)$, it follows that, for $0<\e_1<\e_2$, one has
$[\phi_{\e_2}]\prec[\phi_{\e_1}]\prec[\phi]$ for $[\phi]\in W(A, B)$. Therefore,
$\{[\phi_{\frac1n}]\}_{n\in\IN}$ is seen to be a cofinal
$\prec$-increasing sequence in $[\phi]^\prec\subset W(A,B)$, which shows that
$W(A,B)$ has properties (WO.1), and (WO.2) by the previous lemma. Property
(WO.3) follows from the fact that $(\phi\hoplus\psi)_\eps =
\phi_\eps\hoplus\psi_\eps$ for any c.p.c. order zero maps $\phi,\psi:A\to
M_\infty(B)$. The last property is a consequence of the fact that
$[\phi_\eps]_{\e>0}$ is a cofinal sequence in $[\phi]^\prec$ for any c.p.c. order zero map $\phi:A\to M_\infty(B)$, and that $(\phi\hoplus\psi)_\eps =
\phi_\eps\hoplus\psi_\eps$.
\end{proof}

We now collect the technical results needed to prove that, if $f:A\to A'$ and
$g:B\to B'$ are $*$-homomorphisms between local \Cs-algebras, then the induced
maps $W(f,B):W(A',B)\to W(A,B)$ and $W(A,g):W(A,B)\to W(A,B')$ are morphisms in
the category $\cat W$, in the following two lemmas.

\begin{lemma}\label{lem:morsecond} Let $A,B,B'$ be local \Cs-algebras, $g:B\to
B'$ be a $*$-homomorphism and $\phi:A\to B$ be a c.p.c. order zero map. Then
$(g\circ\phi)_\eps = g\circ\phi_\eps$ for any $\eps > 0$. \end{lemma}
\begin{proof} Let $h_\phi$ and $\pi_\phi$ be such that $\phi = h_\phi\pi_\phi$
as described by Corollary \ref{cor:structure}. Let $g\dd:B\dd\to{B'}\dd$ be the
bidual map of $g$. Then, the decomposition
		$$g\circ\phi = g\dd(h_\phi)(g\dd\circ\pi_\phi).$$
	agrees with the decomposition described in Corollary \ref{cor:structure}.
	Therefore, the result follows from the definition of functional calculus on
	c.p.c. order zero maps between local \Cs-algebras.
\end{proof}

\begin{lemma}\label{lem:morfirst} Let $A,A',B$ be local \Cs-algebras, $f:A\to
A'$ be a $*$-homomorphism and $\phi:A'\to B$ be a c.p.c. order zero map. Then
$\phi_\eps\circ f = (\phi\circ f)_\eps$ for any $\eps>0$.
\end{lemma}
\begin{proof} Observe that, by applying Corollary \ref{cor:structure} at
different stages, the c.p.c. order zero map $\phi\circ f$ can be expressed in
the following equivalent form
		$$\phi\circ f = h_{\phi\circ f}\pi_{\phi\circ f} = h_\phi(\pi_\phi\circ f).$$
	Set $C_\phi := C^*(\phi(A'))$, and let $\{u_\lambda\}_{\lambda\in\Lambda}\subset
	A$ be an increasing approximate unit for $A$. Define the projection $p\in
	C_\phi\dd$ by the strong limit
		$$p := \SOT\lim_\lambda \pi_\phi(f(u_\lambda)),$$
	which commutes with $h_\phi$ and $\pi_\phi\circ f$ by definition. Moreover, $p\pi_\phi(f(a)) = \pi_\phi(f(a))$ for any $a\in A$. A direct computation
	shows that $h_{\phi\circ f} = p h_\phi$, and therefore $\pi_{\phi\circ f} =
	p(\pi_\phi\circ f)=\pi_\phi\circ f$. Since $p$ is a projection,
	$g(h_{\phi\circ f}) = pg(h_\phi)$ for any $g\in C_0((0,1])$. Hence,
	\begin{align*}
		(\phi\circ f)_\eps &= (h_{\phi\circ f}\pi_{\phi\circ f})_\eps\\
			&= (h_{\phi\circ f})_\e(\pi_\phi\circ f)\\
			&= (ph_\phi)_\e(\pi_\phi\circ f)\\
			&= (h_\phi)_\e(\pi_\phi\circ f)\\
			&= \phi_\eps\circ f,
	\end{align*}
	for any $\eps>0$.
\end{proof}

\begin{theorem} The $W(\_,\_)$ is a bifunctor from the category of local \Cs-algebras
to the category $\cat{W}$, contravariant in the first argument and covariant in
the second.
\end{theorem}

\begin{proof} It has already been shown that $W(A,B)$ is in the category $\cat
W$ for any choice of local \Cs-algebras $A$ and $B$. It is left to check that
any $*$-homomorphisms $f:A\to A'$ and $g:B\to B'$ between local \Cs-algebras
induce maps $W(f,B)$ and $W(A,g)$ respectively which are morphisms in $\cat W$.

The continuity of $W(A,g)$ follows from the fact that if $[\psi]\prec
W(A,g)([\phi])$, then there exists $\eps>0$ such that
$\psi\precsim(g^{(\infty)}\circ\phi)_\eps$, which by Lemma \ref{lem:morsecond}
coincides with $\psi\precsim g^{(\infty)}\circ\phi_\eps$. Therefore,
it is enough to take $\phi_\eps$ to witness the continuity, since
$[\phi_\eps]\prec[\phi]$ and $[\psi]\leq W(A,g)([\phi_\eps])$. Let us show now {\rm (WM.2)}. Let
$[\phi]\prec[\psi]$ in $W(A,B)$, i.e. there exists $\quad\eps >0 \quad \text{ such that } \quad \phi\precsim\psi_\eps)$
	Since $W(A,g)$ is order preserving, we must have $W(A,g)([\phi])\leq
	W(A,g)([\psi_\eps])$, whereas, by Lemma \ref{lem:morsecond} we conclude
	that the right-hand side coincides with $[(g^{(\infty)}\circ\psi)_\eps]$,
	whence $W(A,g)([\phi])\prec W(A,g)([\psi])$.
	
	Similarly, using Lemma \ref{lem:morfirst} in place of Lemma \ref{lem:morsecond}, the same
	argument shows that $W(f,B)$ satisfies both properties (WM.1) and (WM.2) as well.
\end{proof}


\subsection{Continuity} As shown in \cite{apt2014}, the category $\cat W$ has
inductive limits and, moreover, the functor $W(\_)$, when defined on the
category of local \Cs-algebras to category $\cat W$, becomes continuous under
arbitrary limits. Therefore, the bivariant functor $W(\_,\_)$ is also continuous
in the second variable trivially whenever the first argument is an elementary
\Cs-algebra. However, in more general cases this property fails, as shown by the
following (counter)examples.

\begin{example} Let $A$ be the \emph{algebraic} CAR algebra, that is the
algebraic direct limit of the inductive sequence
	$$\CC \xrightarrow{\phi_0} M_2
	\xrightarrow{\phi_1} M_4 \xrightarrow{\phi_2}
	\cdots,$$
where the generic connecting map $\phi_n : M_{2^n}\to M_{2^{n+1}}$ is given by
	$$\phi_n(a) = a\oplus a,\qquad\forall a\in M_{2^n}.$$
An element $[\phi]\in W(A,M_k)$ is represented by a c.p.c. order zero map $\phi:A\to M_k(K)\cong K$ with the property that $\phi(1_A)\in K$ commutes with the support $*$-homomorphism $\pi:A\to B(\ell^2(\IN))$ of $\phi$. By functional calculus on $\phi(1)$ one can then find a finite rank projection that commutes with $\pi$. Since there are no $*$-homomorphisms from the CAR algebra to matrix algebras, apart from the trivial one, one sees that $W(A, M_{2^n}) = \{0\}$. However, one verifies
that $W(A,A) = W(A)$ (cf. Section \ref{ssec:ssa}), which is equal to $ \IN_0[\frac12]\sqcup (0,\infty)$ \cite{bpt08}. Continuity
in this case is recovered if one takes the completion $\tilde A$ of $A$, since
simplicity now implies $W(\tilde A, A) = \{0\}$.
\end{example}

A similar computation shows that, in general, the functor is not continuous in
the first argument as well.

\begin{example} Let $A$ be the CAR algebra. Then $W(M_{2^n},\CC) \cong \IN_0$
for any $n$, and the connecting maps are just multiplication by 2 at each step.
Therefore, $\displaystyle\lim_{\longleftarrow} W(M_{2^n},\CC) = \{0\}$, which
coincides with $W(A,\CC) = \{0\}$. But $W(M_{2^n}, K) = \IN_0\cup\{\infty\}$.
Hence, $\displaystyle\lim_{\longleftarrow} W(M_{2^n},\CC) = \{0,\infty\} \neq
W(A,K) = \{0\}$.
\end{example}

\begin{example} Let $A$ be the CAR algebra. Then $W(M_{2^n}, A) \cong W(A) \cong
\IN_0[\frac12]\sqcup (0,\infty)$, with the connecting maps that are now automorphisms of
$\IN_0[\frac12]\sqcup (0,\infty)$. Hence, we have that $\displaystyle\lim_{\longleftarrow}
W(M_{2^n}, A) \cong \IN_0[\tfrac12]\sqcup (0,\infty)$, which can be identified with $W(A, A)\cong W(A)$ (see Section \ref{sec:Examples} for further details).
\end{example}


\subsection{Compact Elements}

In the ordinary theory of the Cuntz semigroup $W(A)$ there is a notion of compact
element: an element $s\in W(A)$ of the Cuntz semigroup of the \Cs-algebra $A$
is compact if $s\ll s$, where $\ll$ denotes the so-called \emph{way-below}
relation we mentioned earlier. It is immediate to check that, according to this definition, every
projection defines a compact element in the Cuntz semigroup. As the natural
bivariant extension of projections are $*$-homomorphisms, we look at a
definition of compact elements for the bivariant Cuntz semigroup such that the
class of every $*$-homomorphism between local \Cs-algebras $A$ and $B$ turns out
to be compact in $WW(A,B)$.

Let $A$ and $B$ be \Cs-algebras. Then, Proposition \ref{prop:MorphismsW} shows that $W(\phi)$ preserves the
\emph{way-below} relation $\ll$ (equivalently the compact containment relation
$\cc$) whenever $\phi$ is a $*$-homomorphism. Hence, considering the stable version of Proposition \ref{prop:MorphismsW}, one gets that $WW(\phi)$ is in the category $\Cu$ (cf. \cite{cei2008}). These
considerations, together with the fact that $*$-homomorphisms over $\CC$
correspond to projections in the target algebra, lead to the following.

\begin{definition} Let $A$ and $B$ be \Cs-algebras. An element $\Phi\in WW(A,B)$
is called compact if there exists a c.p.c. order zero map $\phi$ representing $\Phi$, i.e. $[\phi]=\Phi$,  such that $\Cu(\phi)$ is a morphism in the category $\Cu$.
\end{definition}

It is easy to see that one recovers the usual
definition for compact elements in the ordinary Cuntz semigroup $\Cu(\_)$, when the first algebra is $\CC$. Indeed, consider a \Cs-algebra $B$ and a c.p.c. order zero map $\phi:\CC\to
B\otimes K$. From the structure theorem \sth, one has $\phi(z)=zb$ for any
$z\in\CC$, with $b:=\phi(1)\in B\otimes K$. The induced map
$\Cu(\phi):\IN_0\cup\{\infty\}\to \Cu(B)$ sends $n$ to $n[b]$, with
$[b]\in\Cu(B)$. Since $n=1$ arises from any minimal \emph{projection} in $K$,
one has $1\ll1$ inside $\Cu(\CC)$. Moreover, $\Cu(\phi)$, being the induced map of a
compact element in $WW(\CC,B)$ by hypothesis, preserves the way-below relation; thus, $\Cu(\phi)(1)\ll \Cu(\phi)(1)$, i.e. $[b]\ll[b]$ in $\Cu(B)$.

Other examples of compact elements in the bivariant Cuntz semigroup are given by
the classes of c.p.c. order zero maps that have a $*$-homomorphism as a
representative. This follows from the fact that Cuntz-equivalent c.p.c. order
zero maps induce the same morphism at the level of the Cuntz semigroups and that
$*$-homomorphisms preserve the relation $\ll$. The following is an easy example where the above happens.

\begin{example} Let $A$ be the unital \Cs-algebra $C([0,1])$. Take $\pi:A\to A$ to be the identity map, and $h\in A_+$ the continuous map which takes value $h(0)=1/2$ and $h(1)=1$, and linear in between. The map $\phi(\_):=h\pi(\_)$ defines a c.p.c. order zero map from $A$ to $A$, which is not a $*$-homomorphism. However, since $h$ is an invertible element in $A$, $\phi$ is Cuntz-equivalent to $\pi$, as $\phi(a) = h^{\frac12}\pi(a)h^{\frac12}$ and $h^{-\frac12}\phi(a)h^{-\frac12} = \pi(a)$ for any $a\in A$.
\end{example}

With the above considerations in mind, we introduce the following notation
	$$V(A,B) := \{\Phi\in W(A,B)\ |\ \Phi=[\pi]\text{ for some $*$-homomorphism $\pi$}\}.$$
In particular, thanks to \cite[Lemma 2.20]{apt2009}, one sees that $V(\CC,B)\cong V(B)$, i.e. the Murray-von Neumann semigroup of $B$, whenever $B$ is a stably finite local \Cs-algebra. In analogy with the contents of Section 2.4.2 of \cite{apt2009}, we denote by $W(A,B)_+$ all the other elements of $W(A,B)$ that do not belong to $V(A,B)$. Hence, any c.p.c. order zero map $\phi:A\to \Mi(B)$ such that $[\phi]\in W(A,B)_+$ will be called \emph{purely c.p.c. order zero}. We then have a decomposition of the bivariant Cuntz semigroup as the disjoint union
	$$W(A,B) = V(A,B)\sqcup W(A,B)_+$$
for any pair of local \Cs-algebras $A$ and $B$. Furthermore, $W(A,B)_+$ is \emph{absorbing}, in the sense that $\Phi+\Psi\in W(A,B)_+$ whenever $\Phi\in V(A,B)$ and $\Psi\in W(A,B)_+$.

The next result shows that $V(A,B)$ corresponds, under special circumstances, to the set of compact elements in $WW(A,B)$.

\begin{theorem} Let $A$ and $B$ be \Cs-algebras, with $A$ unital and $B$ stably finite. Then $V(A,B)$ is the subsemigroup of all the compact elements of $WW(A,B)$.
\end{theorem}
\begin{proof} Clearly every element in $V(A,B)$ is compact in $WW(A,B)$., so let us prove the converse. Let $\Phi\in WW(A,B)$ be a compact element. Then, there exists a representative of the form $\phi\otimes\id_K\in\Phi$, with $\phi:A\to B\otimes K$, such that $\Cu(\phi)$ preserves the way-below relation (cf. Remark \ref{rem:Repre}). Since $1_A\otimes e\in A\otimes K$ is a projection, $e$ being a minimal projection in $K$, its class in $\Cu(A)$ is a compact element. Therefore, $\Cu(\phi)([1_A\otimes e])\ll\Cu(\phi)([1_A\otimes e])$ implying that $[\phi(1_A)]\ll[\phi(1_A)]$ in $\Cu(B)$. Since $B$ is stably finite, the positive element $h:=\phi(1_A)$ is then Cuntz-equivalent to a projection $p$ by \cite[Theorem 3.5]{bc}, e.g. its support projection $p=p_h$. Hence, $h^{\frac12}p_h = h^{\frac12}$, and there exists $\seq x\subset B$ such that $x_nh^{\frac12}\to p$ in norm, viz.
	$$x_n:=(h+\tfrac1n)^{-1}h^{\frac12}.$$
Thus, since $h^{\frac12}\pi_\phi(a)h^{\frac12}=\phi(a)$, one has $\norm{x_n^*\phi(a)x_n-\pi_\phi(a)}\to 0$ for any $a\in A$, which shows that $\phi$ is Cuntz-equivalent to its support $*$-homomorphism.
\end{proof}

The above theorem can be regarded as the bivariant version of the analogous result for the Cuntz semigroup (\cite[Theorem 3.5]{bc}).


\section{The Composition Product}

We start this section by recalling some further facts about c.p.c. order zero
maps that are used later on to introduce a composition product on the bivariant
Cuntz semigroup, which resembles the Kasparov's composition product in KK-Theory.
This product is used in the following sections to define a notion of $WW$-equivalence between C*-algebras, which will be crucial to establish our main classification result, i.e. Theorem \ref{thm:classification}. We define this new notion of equivalence in a spirit similar to  KK-equivalence in KK-theory.

Let $A$, $B$, $C$ be local \Cs-algebras, and let $\phi: A\to M_\infty(B)$ and
$\psi:B\to M_\infty(C)$ be any c.p.c. order zero maps. With $\psi^{(\infty)}$ denoting $\phi\otimes\id_{M_\infty}$, the
composition
	$$\phi\cdot\psi := \psi^{(\infty)}\circ\phi$$
defines a c.p.c. order
zero map from $A$ to $M_\infty(C)$. One defines a composition product among
elements of \emph{composable} bivariant Cuntz semigroups by just pushing the
above composition product forward to the corresponding classes.

\begin{proposition}[Composition product] Let $A$, $B$ and $C$ be separable local
\Cs-algebras. The binary map $W(A,B)\times W(B,C)\to W(A,C)$ given by
		$$[\phi]\cdot[\psi] := [\phi\cdot\psi]$$
	is well-defined. We call this map the \emph{composition product} for the
	bivariant Cuntz semigroup.
\end{proposition}
\begin{proof} Let $\phi,\phi':A\to M_\infty(B)$, $\psi,\psi':B\to M_\infty(C)$
be c.p.c. order zero maps such that $\phi\precsim\phi'$ and $\psi\precsim\psi'$.
Since the latter condition implies $\psi^{(\infty)}\precsim{\psi'}^{(\infty)}$ by Corollary \ref{cor:stabilitycse},
it follows from Proposition \ref{prop:compositions} that
$\phi\cdot\psi\precsim\phi\cdot\psi'$ and $\phi\cdot\psi\precsim\phi'\cdot\psi$.
\end{proof}

\begin{corollary} The composition product on the bivariant Cuntz semigroup and
its order structure are compatible, in the sense that, if
$\phi,\phi',\psi,\psi'$ are as in the above proposition, then
$\phi\cdot\psi\precsim\phi'\cdot\psi'$.
\end{corollary}

It is clear that, for any local \Cs-algebra $A$, $W(A,A)$ has a natural semiring
structure. In particular, the class of the embedding $\iota_A:A\to M_\infty(A)$
in $W(A,A)$ provides a unit $[\iota_A]$.

\begin{example} With $A=\CC$, we obtain the semigroup $W(\CC,\CC) = W(\CC) =
\IN$. It is an easy exercise to verify that, if $[\phi],[\psi]\in W(\CC,\CC)$,
the product for the corresponding positive elements $[h_\phi],[h_\psi]\in
W(\CC)$ is given by the tensor product $[h_\phi\otimes h_\psi]$. Therefore, the
composition product corresponds to the ordinary product between natural numbers
in $\IN$.
\end{example}

\begin{definition}[Invertible element in $W$] Let $A$ and $B$ be separable
\Cs-algebras. An element $\Phi\in W(A,B)$ is said to be invertible if there
exists $\Psi\in W(B,A)$ such that $\Phi\cdot\Psi = [\iota_A]$ and $\Psi\cdot\Phi
= [\iota_B]$.
\end{definition}

\begin{definition}[$W$-equivalence] Two separable \Cs-algebras $A$ and $B$ are
$W$-equivalent if there exists an invertible element in $W(A,B)$.
\end{definition}

The semigroup $WW(A,B)$ inherits the composition product directly from its definition, i.e. 
$WW(A,B)=W(A\otimes K, B\otimes K)$. However, one can give an equivalent definition
where it takes the form of a genuine composition, since there is no need of
considering matrix ampliations in this case. Thus, if $A$, $B$ and $C$ are
separable \Cs-algebras, and $\phi:A\to B$ and $\psi:B\to C$ are c.p.c. order
zero maps, then one can set
	$$\phi\cdot\psi := \psi\circ\phi.$$
Thus, $WW(A,A)$ has a natural semiring structure too, and the unit is seen to be
represented by the class of the identity map on $A\otimes K$, viz.
$[\id_{A\otimes K}]$.

\begin{definition}[Invertible element in $WW$] Let $A$ and $B$ be separable
\Cs-algebras. An element $\Phi\in WW(A,B)$ is said to be invertible if there
exists $\Psi\in WW(B,A)$ such that $\Phi\cdot\Psi = [\id_{A\otimes K}]$ and
$\Psi\cdot\Phi = [\id_{B\otimes K}]$.
\end{definition}

\begin{definition}[$WW$-equivalence] Two \Cs-algebras $A$ and $B$ are
$WW$-equivalent if there exists an invertible element in $WW(A,B)$.
\end{definition}

As said before, an application of the above definitions to the problem of classification of
\Cs-algebras is given in Section \ref{sec:classification}, where a slightly
stronger notion of invertibility, together with the notion of scale for the
bivariant Cuntz semigroup, is introduced.

Another important aspect of the product introduced in this section is the map $W(A,B)\to\Hom(W(A),W(B))$, and its stabilized counterpart, i.e. $WW(A,B)\to\Hom(\Cu(A),\Cu(B)))$. This arises from the observation in \cite{wz2009} that any c.p.c. order zero map induces a map a the level of the Cuntz semigroups and that the correspondence $[\phi]\mapsto W(\phi)$ is well-defined. The same correspondence can be recovered by means of the composition product of this section. Indeed, one can define a map from $W(A,B)$ to $\Hom(W(A),W(B))$ by exploiting the isomorphism $W(A)\cong W(\CC,A)$ for any local \Cs-algebra $A$. In particular, one can set $\gamma: W(A,B)\to\Hom(W(A),W(B))$ to be the map given by
	$$\gamma(\Phi)(s) := s\cdot\Phi,\qquad\forall s\in W(A)\cong W(\CC,A),\Phi\in W(A,B).$$
Notice that $\gamma(\Phi)(s)$ belongs to $W(B)$ by the composition product, and, moreover, $\gamma(\Phi)(s)= W(\phi)(s)$ for any representative $\phi$ of $\Phi$ (cf. Proposition \ref{prop:MorphismsW}).

It seems as a natural question to wonder whether the map $\gamma$ is surjective in general, or if there exists a special class of \Cs-algebras for which this is the case. As interesting as this question is, which could lead to a \emph{Cuntz analogue} of the UCT class, this aspect will be addressed elsewhere.


\section{Examples}\label{sec:Examples}

\subsection{Purely infinite \Cs-algebras}

In this section we determine the bivariant Cuntz semigroup, $WW(A,B)$, whenever $B$ is a Kirchberg algebra and $A$ is a unital and exact \Cs-algebra, and both algebras are separable.  Recall that a Kirchberg algebra is a purely infinite, simple and nuclear \Cs-algebra (see e.g. \cite{rordam}).

We will use the following fundamental approximation result for u.c.p. maps on unital Kirchberg algebras (\cite[Corollary 6.3.5]{rordam}).

\begin{lemma}\label{fundamental}
Let $B$ be a unital Kirchberg algebra, $\rho: B \to B$ be a u.c.p. map, $F \subset B$ be a finite subset and $\e >0$. Then, there exists an isometry $s \in B$ such that $\| s^* b s - \rho(b) \| \leq \e \| b \|$ for all $b \in F$.
\end{lemma}

We also need the following interpolation lemma for u.c.p. maps which follows from Arveson's extension theorem. Recall that an operator system in a unital \Cs-algebra $A$ is a closed self-adjoint subspace of $A$ containing the identity of $A$. 

\begin{lemma}\label{approx-extension}
Let $B$ be a nuclear \Cs-algebra, $E\subset B$ be a finite dimensional operator system, $\eta : E \to B$ be a u.c.p. map, and $\e >0$. Then, there exists a u.c.p. map $\tilde{\eta} : B \to B$ such that $\| \tilde{\eta} (x) - \eta (x) \| \leq \e \|x\|$ for all $ x \in E$.
\end{lemma}
\begin{proof} The nuclearity of $B$ makes the inclusion $E \hookrightarrow B$ a
nuclear map. Therefore, for any $\eps > 0$, there is an $n\in\IN$ and u.c.p. maps
$\rho:E\to M_n$ and $\phi:M_n\to B$ such that
	$$\norm{\phi\circ\rho-\eta}\leq\eps.$$
By Arveson's extension theorem, the map $\rho$ admits a u.c.p. extension to $B$, i.e.
there exists $\tilde\rho : B\to M_n$ u.c.p. such that $\tilde\rho|_E = \rho$.
The situation is depicted in the following diagram
$$\begin{tikzcd}
	E\ar[hookrightarrow]{d}\ar{rr}{\eta}\ar{dr}{\rho} & & B\\
	B\ar{r}{\tilde\rho} & M_n\ar{ur}{\phi} & ,
\end{tikzcd}$$
which commutes up to $\eps$. By setting $\tilde\eta := \phi\circ\tilde\rho$ we
then have $\norm{\tilde\eta|_E-\eta}\leq\eps$.
\end{proof}

Using these two lemmas, we can show a subequivalence result for order zero maps into Kirchberg algebras. First we introduce some notation: given an order zero map $\phi :A \to B$, we denote, as before,  its decomposition by $h \pi$, where $\pi : A \to \mathcal{M} (C^*(\phi(A)))$ and $h \in \mathcal{M} (C^*(\phi(A)))$. In fact, the range of $\pi $ lies in $\mathcal{M}(B_{\phi})$, where $B_{\phi}$ is the hereditary subalgebra $\overline{\phi(A) B \phi(A)}$. Let $g_{\e}$ be the continuous function on $[0,\infty)$ which is 0 on $ [0,\e/2)$, 1 on $[\e , \infty)$, linear otherwise, and let $h_{\e}= g_{\e}(h)$ and $\phi_{\e} = h_{\e} \pi$. It is not hard to see that there exists a continuous positive function $k_{\e}$ vanishing on $[0,\e/2]$ such that $t k_{\e} (t) = g_{\e}(t)$ so that $\bar{h}_{\e}:= k_{\e}(h)$ satisfies $\bar{h}_\eps h=h_{\e} $ and $\|\bar{h}_\eps \| \leq \frac{1}{\e}$.  Hence, $h_{2\e}^{1/n}$ converges to an open projection $p_{\e}$ in $B$ as $n \to \infty$ and $\pi_{\e} = p_{\e} \pi$ can be regarded as a $*$-homomorphism from $A$ to the multiplier algebra $\mathcal M( B_{\e} )$, where $B_{\e} = p_{\e} B p_{\e} \cap B = h_{2\e} B h_{2\e}$. Now $h_{\e} \in B_{\e/2}$ so that $\phi_{\e}= h_{\e}\pi_{\e/2} $ is a decomposition of this order zero map. There is a canonical c.p. map from $\phi_{\e}(A)$ to $\pi_{\e}(A)$  given by $ x \mapsto p_{\e} x p_{\e}$. Note that $p_{\e} \phi_{\e}(a) p_{\e} = \pi_{\e}(a)$ for all $a \in A$. 

Observe further that $\phi$ is injective if and only if $\pi$ is injective, and, in this case, 
	$$\lim_{\e \to 0^+} \| \pi_{\e} (a) \| = \|a\|$$
for every $a \in A$. Moreover, $\pi_{\e}(A)$ may be identified with $A/\ker \pi_{\e}$. If $A$ is nuclear, then this quotient map admits a completely positive lift, and, if $A$ is only exact, then the quotient map $ A \to A/\ker \pi_{\e}$ has the local lifting property, i.e. given any finite dimensional operator system $E \subset   A/\ker \pi_{\e}$ there exists a u.c.p. map $\lambda_{\e} : E\to A$ with $\pi_{\e} \circ \lambda_{\e} = \textup{id}_{E}$. In case $A$ happens to be simple, then $\ker \pi_{\e}= 0$ for sufficiently small $\e$, so that the existence of this lift is obvious in this situation.

\begin{lemma}\label{ApproxUnital}
Let $A$ be a unital separable exact C*-algebra, $B$ be a unital Kirchberg algebra, $\phi_{1}, \phi_2 :A \to B$ be c.p.c. order zero maps and assume that $\phi_1$ is injective. Then, $\phi_2 \precsim \phi_1$, i.e., there exists a sequence of elements $\seq b$ in $B$ such that $\|b_n^*\phi_1(a)b_n-\phi_2(a)\|\to 0$ for all $a\in A$.
\end{lemma}

\begin{proof} We will show that given a finite dimensional operator system $E \subset A$ and $\e >0$, there exists $b \in B$ such that $\| b^* \phi_1(e) b - \phi_2(e) \| \leq \e$ for all $ e \in E$ with $\| e \| \leq 1$, which suffices to conclude the proof. Continuing with the notation introduced above, write $\phi_{k}= h_k\pi_k$ for $k=1,2$, and $h_{k,\e}$, $\pi_{k,\e}$ for $k=1,2$ and $\e>0$. We assume that for every $\delta>0$, we have $h_{k,\delta} \neq 1$ for $k=1,2$; otherwise the proof works with minor modifications. We further assume that $A$ is simple, and describe below the necessary changes in the non-simple case. 

Choose $\delta$ small enough so that $\pi_{1,\delta}$ is non-zero; hence, it is injective with inverse $\lambda_{\delta} = \pi_{1,\delta}^{-1}: \pi_{1,\delta}(A) \to A$. We define a u.c.p. map $\rho_{1,\delta}$ from the operator system  
	$$E_{1,\delta}:=\phi_{1,\delta}(E) + \CC 1= \phi_{1,\delta}(E) + \CC (1-h_{1,\delta})$$
to $A$ as follows:
$$
\rho_{1,\delta} (\phi_{1,\delta}(e) + \lambda (1- h_{1,\delta})) = \lambda_{\delta} (p_{\delta}   (\phi_{1,\delta}(e) + \lambda (1- h_{1,\delta})) p_{\delta}).
$$ 
Since $1-h_{1,\delta}$ and $p_{\delta}$ are orthogonal, this is equal to $e$. Now fix any state $\omega$ on $A$ and consider the unital modification of $\phi_{2,\delta}$ given by $a \mapsto \phi_{2,\delta}(a) + \omega (a) (1 -h_{2,\delta})$. The composition of $\rho_{1,\delta}$ followed by this map gives a u.c.p. map 
$
\eta: E_{1,\delta} \to B.
$
By Lemma \ref{approx-extension}, we can find $\tilde{\eta}: B \to B$ u.c.p. such that for all $x \in E_{1,\delta}$
$$
\|\tilde{\eta} (x) - \eta (x) \| \leq \tfrac{\e}{6} \| x\|.
$$ 
In particular,  
$$
\|\tilde{\eta} (x) - \eta (x) \| \leq \tfrac{\e}{6} \text{ for } \| x \| \leq 1.
$$ 
Since $\eta (1-h_{1,\delta}) =0$, we have 
$$
\| \tilde{\eta} (1-h_{1,\delta}) \| \leq \tfrac{\e}{6}.
$$ 
By Lemma \ref{fundamental} we can find an isometry $s \in B$ such that 
$$
\| s^* xs - \tilde{\eta}(x) \| \leq \tfrac{\e}{6} \|x \|.
$$
Therefore, it is not hard to see that 
$$
\| s^* (\phi_{1,\delta} (e) + \lambda (1-h_{1,\delta} ) )s - s^* \phi_{1,\delta} (e) s \| \leq \tfrac{2\e}{6} 
|\lambda |  \leq \tfrac{\e}{3}.
$$
It follows that 
$$
\| h_2^{1/2} s^* \phi_{1,\delta} (e) s h_2^{1/2} - h_2^{1/2}  \phi_{2, \delta}(e) h_2^{1/2} \| \leq \tfrac{2\e}{3},
$$
whenever $\ e \in E$ and $\| e \| \leq 1$. Next we remark that $\|h_2(1-h_{2,\delta})\| = \| h_2 - h_2 h_{2,\delta} \|  \leq \delta$. Thus, choosing $\delta \leq \tfrac{\e}{6}$, we find that 
$\|  h_2^{1/2}  \phi_{2, \delta} (e) h_2^{1/2} - \phi_2(e) \| \leq \tfrac{\e}{6} \|e\|$ for all $ e \in E$. Writing $\phi_{1,\delta} = \bar{h}_{1,\delta} \phi_1$,  we finally get
$$
 \| h_2^{1/2} s^* \bar{h}_{1,\delta}^{1/2} \phi_{1} (e) \bar{h}_{1,\delta}^{1/2} s h_2^{1/2} - \phi_{2}(e)  \| \leq \e,
$$  
for $\|e \| \leq 1$, so that $b=\bar{h}_{1,\delta}^{1/2} s h_2^{1/2}$ is as required. 

If $A$ is not simple we have to modify the proof as follows: replace the lift $\pi_{1,\delta}^{-1}$ by a local lift $\lambda_{\delta} : E/\ker \pi_{1,\delta} \to A$ depending on $E$ and $\delta$. We have $\lambda_{\delta} \circ  \pi_{1,\delta} (e) = e + j_{\delta} (e)$, where $j_{\delta}(e) \in J_{\delta} :=\ker \pi_{1,\delta}$. Note that $\| j_{\delta} (e) \| \leq 2 \|e \|$ and that for any bounded linear functional $\omega$ we have $\|\omega |_{J_{\delta}}\| \to 0$ as $\delta \to 0$\footnote{This follows simply from the fact that $\bigcap_{\delta} J_{\delta} =\{ 0\}$.}. The same is true for every u.c.p. map into a finite dimensional \Cs-algebra. Using nuclearity of the map $\phi_2: A \to B$ we can conclude the proof similarly to the first part.
\end{proof}

\begin{remark} Under the assumptions of Lemma \ref{ApproxUnital}, all injective order zero maps $\phi : A \to B$ are  Cuntz-equivalent. In particular for $A$ simple there can be at most two such classes. On the other hand, given a Kirchberg algebra $B$ every separable exact \Cs-algebra $A$ embeds into $B$. (By Kirchberg's embedding theorem there exists an embedding of $A$ into $\mathcal{O}_2$. Combine this with a  non-unital  embedding of $\mathcal{O}_2$ into $\mathcal{O}_{\infty} \subset \mathcal{O}_{\infty} \otimes B \cong B$, which exists by \cite[4.2.3]{rordam}.)

For $A$ not necessarily simple, Lemma \ref{ApproxUnital} shows that two order zero maps $\phi, \psi : A \to B$ are equivalent if and only if they have the same kernel. Hence, the Cuntz-equivalence classes are labelled by the ideal space of $A$. To show that every ideal $J$ of $A$ occurs one can apply the same argument as in the simple case to obtain an embedding of the quotient $A/J$ into $B$. (Note that $A/J$ is an exact \Cs-algebra.)
\end{remark}

\begin{theorem}
Let $A$ be a unital separable simple exact C*-algebra and $B$ be any Kirchberg algebra, then $WW(A,B)\cong\{0,\infty\}$. If $A$ is only unital, separable and exact, then $WW(A,B)$ is labelled by the ideal space of $A$, that is there is a one-to-one correspondence between elements  $[\phi] \in  WW(A,B)$ and (closed two-sided) ideals in $A$ given by $J_{\phi} = \ker \phi = \ker \pi_{\phi}$. Here the addition of elements in $WW(A,B)$ corresponds to the intersection of the corresponding ideals.

\end{theorem}
\begin{proof}
Every non-unital Kirchberg algebra is stable and can be written as a tensor product of a unital Kirchberg algebra and $K$. Applying this to $B \otimes K$ 
we may assume $B$ to be a unital Kirchberg algebra and $B\otimes K$ the target of $WW(A,B\otimes K)$. Then, we only need to remark that any order zero map $\phi: A \to B \otimes K$ is Cuntz equivalent to an order zero map with range in $B \otimes e_{0,0}$, since then we are able to apply Lemma \ref{ApproxUnital}. To prove this, choose a sequence of pairwise orthogonal projections $p_0, p_1, p_2, \ldots$ each of which are Murray-von Neumann equivalent to $1 \in B$. Such a sequence  exists by an easy induction argument. Now, let
	$$B_0 = \lim_{n \to \infty} (p_0 + \ldots + p_n)B (p_0 + \ldots + p_n) \subset B \otimes e_{0,0}.$$
By assumptions, there exists a sequence of partial isometries $v_n \in B \otimes K$ such that $v_n v_n^* = (p_0 + \ldots + p_n) \otimes e_{0,0}$ and $ v_n^* v_n = 1 \otimes (e_{0,0} +  \dots + e_{n,n}) $ and $v_{n+1} $ extends $v_n$. It follows that $v_n \phi v_n^*$ converges point-wise to an order zero map $\phi_0 : A \to B_0 \otimes e_{0,0} \subset B \otimes e_{0,0}$, and therefore $\seq v$ implements a Cuntz equivalence between $\phi$ and $\phi_0$.
\end{proof}


\subsection{Strongly Self-absorbing \Cs-algebras\label{ssec:ssa}}

We now establish a property of the bivariant Cuntz semigroup when its arguments
are tensored by a strongly self-absorbing \Cs-algebras. As an application of
this result we give some explicit computations of some bivariant Cuntz
semigroups.

As observed in \cite{tw}, every strongly self-absorbing \Cs-algebra $\mathcal D$
is a unital \Cs-algebra, different from $\CC$, for which there exists a unital
$*$-homomorphism $\gamma:\mathcal D\otimes \mathcal D\to \mathcal D$ satisfying
to $\gamma\circ(\id_{\mathcal D}\otimes 1_{\mathcal D})\au\id_{\mathcal D}$ (cf.
\cite[Proposition 1.10(i)]{tw}). Hence, in what follows, we shall assume that
every strongly self-absorbing \Cs-algebra $\mathcal D$ comes equipped with such
a map $\gamma$. Moreover, we adopt \cite[Definition 1.1]{tw} as the definition
of approximate unitary equivalence $\au$ between c.p.c. order zero maps. We also
refer to \cite[Proposition 1.2]{tw} for some well-known facts about this
relation. We start describing the connection between approximate unitary
equivalence and Cuntz comparison.

\begin{proposition}\label{prop:au} Let $A$ and $B$ be \Cs-algebras, and let
$\phi,\psi:A\to B$ be c.p.c. order zero maps such that $\phi\au\psi$. Then,
$\phi\sim\psi$ in $W(A,B)$.
\end{proposition}
\begin{proof} By cutting down the sequence of unitaries $\seq u\subset\mathcal
M(B)$ by an approximate unit $\seq e\subset B$, if necessary, one gets a
sequence in $B$ that witnesses the sought Cuntz equivalence.
\end{proof}

As a consequence of the above result and Proposition \ref{prop:compositions}, we
have that, if $A$, $B$ and $C$ are \Cs-algebras, and $\phi,\psi:A\to
B$, $\eta:B\to C$ are c.p.c. order zero maps such that $\phi\au\psi$, then
$\eta\circ\phi\sim\eta\circ\psi$.

\begin{lemma}\label{lem:autensor} Let $A$, $B$, $C$ and $D$ be \Cs-algebras, $D$
unital, and let $\phi,\psi:A\to D$, $\eta:B\to C$ be c.p.c. order zero maps,
such that $\phi\au\psi$. Then, $\eta\otimes\phi\au\eta\otimes\psi$.
\end{lemma}
\begin{proof} Let $\seq u\subset D$ be the sequence that witnesses the
approximate unitary equivalence $\phi\au\psi$. Since $\mathcal M(C)\otimes
D\subset\mathcal M(C\otimes D)$, the sequence $\{1_{\mathcal M(C)}\otimes
u_n\}_{n\in\IN}\subset\mathcal M(C\otimes D)$ witnesses
$\eta\otimes\phi\au\eta\otimes\psi$.
\end{proof}

\begin{theorem} Let $A$, $B$ be local \Cs-algebras and $\mathcal D$ be a
strongly self-absorbing \Cs-algebra. The following isomorphism holds,
		$$W(A\otimes \mathcal D, B\otimes \mathcal D)\cong W(A, B\otimes \mathcal
		D).$$ 
\end{theorem}
\begin{proof} By the functoriality of $W(\_,\_)$, it follows that there is an isomorphism $W(A, B\otimes \mathcal D)\cong W(A,B\otimes \mathcal D\otimes \mathcal D)$ induced by the isomorphism between $\mathcal D$ and $\mathcal D\otimes \mathcal D$. Hence, it is enough to show that $W(A,B\otimes \mathcal D)$ is isomorphic to $W(A\otimes \mathcal D, B\otimes \mathcal D\otimes \mathcal D)$. To do so, let $\gamma$ be the unital $*$-homomorphism associated to $\mathcal{D}$. We claim that the maps\footnote{here $\id_K$ is used instead of the identity map on $M_\infty\subset K$.}
		$$\xymatrix@C=2em@R=0pt{%
			W(A,B\otimes \mathcal D)\ar[r] & W(A\otimes \mathcal D, B\otimes \mathcal D\otimes \mathcal D)\\
			[\phi]\ar@{|->}[r] & [\phi\otimes\id_{\mathcal D}]
		}$$
	and
		$$\xymatrix@C=2em@R=0pt{%
			W(A\otimes \mathcal D, B\otimes \mathcal D\otimes \mathcal D)\ar[r] & W(A,B\otimes \mathcal D)\\
			[\psi]\ar@{|->}[r] & [(\id_{B\otimes K}\otimes\gamma)\circ\psi\circ(\id_A\otimes1_{\mathcal D})]
		}$$
	are mutual inverses. Indeed, by a repeated use of Lemma \ref{lem:autensor}, we have
		\begin{align*}
			(\id_{B\otimes K}\otimes\gamma)\circ(\phi\otimes\id_{\mathcal D})\circ(\id_A\otimes 1_{\mathcal D})
				&=(\id_{B\otimes K}\otimes\gamma)\circ(\id_{B\otimes K}\otimes\id_{\mathcal D}\otimes1_{\mathcal D})\circ\phi\\
				&\au(\id_{B\otimes K}\otimes\id_{\mathcal D})\circ\phi\\
				&=\phi,
		\end{align*}
	and
		\begin{align*}
			((\id_{B\otimes K}\otimes\gamma)\circ\psi\circ(\id_A&\otimes 1_{\mathcal D}))\otimes\id_{\mathcal D} =\\
				&= (\id_{B\otimes K}\otimes\gamma\otimes\id_{\mathcal D})\circ(\psi\otimes\id_{\mathcal D})\circ(\id_A\otimes1_{\mathcal D}\otimes\id_{\mathcal D})\\
				&\sim(\id_{B\otimes K}\otimes\gamma\otimes\id_{\mathcal D})\circ(\psi\otimes\id_{\mathcal D})\circ(\id_A\otimes\id_{\mathcal D}\otimes1_{\mathcal D})\\
				&=(\id_{B\otimes K}\otimes\gamma\otimes\id_{\mathcal D})\circ(\psi\otimes1_{\mathcal D})\\
				&=(\id_{B\otimes K}\otimes\gamma\otimes\id_{\mathcal D})\circ(\id_{B\otimes K}\otimes\id_{\mathcal D}\otimes\id_{\mathcal D}\otimes 1_{\mathcal D})\circ\psi\\
				&\au(\id_{B\otimes K}\otimes\gamma\otimes\id_{\mathcal D})\circ(\id_{B\otimes K}\otimes\id_{\mathcal D}\otimes1_{\mathcal D}\otimes\id_{\mathcal D})\circ\psi\\
				&\au(\id_{B\otimes K}\otimes\id_{\mathcal D}\otimes\id_{\mathcal D})\circ\psi\\
				&=\psi.
		\end{align*}
	At the level of the Cuntz semigroups all the above equivalences reduce to
	equality by Proposition \ref{prop:au}.
\end{proof}
\begin{corollary} Let $A$, $B$ be \Cs-algebras and $\mathcal D$ be a strongly
self-absorbing C*-algebra. Then it follows,
		$$WW(A\otimes \mathcal D, B\otimes \mathcal D)\cong WW(A,B\otimes \mathcal
		D).$$
\end{corollary}

As already mentioned, the above result can be used to explicitly compute some
bivariant Cuntz semigroups.

\begin{example} Let $U$ be any UHF algebra of infinite type. Then $W(U, U) \cong
WW(U)$. In particular, if $\mathcal Q$ is the universal UHF algebra, then
$W(U,\mathcal Q)\cong W(\mathcal Q)$.
\end{example}

\begin{example} Let $A$ be any separable \Cs-algebra, and let $\mathcal Z$ be
the Jiang-Su algebra. Then $W(\mathcal Z, A\otimes\mathcal Z) \cong
W(A\otimes\mathcal Z)$. In particular, if $A$ is $\mathcal Z$-stable, then
$W(\mathcal Z, A\otimes\mathcal Z)\cong W(A)$.
\end{example}

\begin{remark} Notice that the above examples also hold considering the functor
$WW(\_,\_)$ and the stable Cuntz semigroup respectively.
\end{remark}


\subsection{Cuntz Homology for Compact Hausdorff Spaces} In this section we give
an explicit computation of the bivariant Cuntz semigroup $WW(C(X), \CC)$, which
can be regarded as a first step towards an analogue of the K-homology for
compact Hausdorff spaces in the setting of the Cuntz theory. Throughout this
section, we let $X$ denote a compact and metrizable Hausdorff space, unless
otherwise stated.

Recall that if $\phi:A\to B$ is a c.p.c. order zero map between
\Cs-algebras, its kernel coincides with that of its support $*$-homomorphism
$\pi_\phi$. In particular, when $A=C(X)$, then $\ker\phi$ can be identified with
the closed subspace of $X$ on which every function in $\ker\phi$ vanishes.

In this setting we regard $\pi_{\phi} : C(X) \to B(H)={\mathcal M} (K(H))$ as a non degenerate representation and $h$ as a compact operator 
supported on a certain Hilbert space $H$. ($H$ is the subspace on which $h=\phi(1)$ is supported.) By the Spectral Theorem $H$ decomposes into a direct integral $H =\int_X^{\oplus} H_x d\mu (x)$.
If $H$ happens to be finite dimensional then $\pi_{\phi}$ will be unitarily equivalent to the evaluation representation on a finite sequence 
$\bigoplus_{i=1}^n ev_{x_i}$. Since the compact operator $h$ commutes with $\pi_{\phi}$ the finite dimensional eigenspaces of $h$ reduce 
$\pi_{\phi}$ and so $\pi_{\phi} $ is unitarily equivalent to $\bigoplus_{i \in \IN} ev_{x_i}$, where $(x_i)$ is a sequence in $X$ with finitely or infinitely many repetitions.

\begin{definition} 
The spectrum $\sigma(\phi)$ of a c.p.c. order zero map $\phi:
C(X) \to K$ is the closed subset of $X$ associated to the kernel of $\phi$, i.e.
	$$\sigma(\phi) := \{x\in X\ |\ f(x) = 0\ \forall f\in\ker\phi\}.$$
\end{definition}

It is convenient to separate the set of isolated points of the spectrum from the
set of accumulation points. The former will be denoted by $\sigma_i(\phi)$,
while the latter is defined as
$\sigma_\ess(\phi):=\sigma(\phi)\smallsetminus\sigma_i(\phi)$. Our notation
follows the usual definition of the essential spectrum of a normal operator,
with the only exception that here we do not include isolated points with
infinite multiplicity in it. If $x$ is an isolated point from a subset $C$ of
$X$, then there exists a neighborhood $U$ of $x$ that does not contain other
points of $C$. By Urysohn's lemma, one finds a continuous function
$\tilde\chi_{\{x\}}\in C(X)$ that vanishes on the outside of $U$ and such that
$\tilde\chi_{\{x\}}(x)=1$. We use this fact to provide continuous indicator
functions $\tilde\chi$ for isolated points of subsets in the relative topology.

We now define the notion of multiplicity functions for a c.p.c order zero map,
and relate them to the semigroup $WW(C(X),\CC)$.

\begin{definition}[Multiplicity function] Let $\phi:C(X)\to K$ be a c.p.c. order
zero map. The multiplicity function $\nu_\phi$ of $\phi$ is the map from $X$ to
$\tilde\IN_0:=\IN_0\cup\{\infty\}$ given by
	$$\nu_\phi(x) = \begin{cases}
		0 & x\not\in\sigma(\phi)\\
		\infty & x\in\sigma_\ess(\phi)\\
		\rk \pi_\phi(\tilde\chi_{\{x\}}) & x\in\sigma_i(\phi),
	\end{cases}$$
where $\rk$ denotes the rank of an element in $K$.
\end{definition}

Let $\phi:C(X)\to K$ be a c.p.c. order zero map with decomposition $\phi = \pi_{\phi} h$ according to \sth.
As already pointed out, up to unitary equivalence, $\pi_{\phi}$ can be taken of the form
	$$\pi_\phi(f)=\bigoplus_{n\in\IN}f(x_n),$$
where  $\seq
x\subset\sigma(\phi)$ is a dense sequence  with possible repetitions.
For isolated points the multiplicity function $\nu_\phi$ associated to $\phi$ gives the number of
occurrence of every $x_n$ in the sequence, whereas an
accumulation point of $\sigma (\phi)$ has, by definition, infinite multiplicity.
Up to unitary equivalence, one can split
$\pi_\phi$ into
	\begin{equation}\label{eq:homdec}
		\pi_\phi = \pi_{\phi,i}\hoplus\pi_{\phi,\ess},
	\end{equation}
where
	$$\pi_{\phi,i}(f) := \bigoplus_{x\in\sigma_i(\phi)}f(x)\Id_{\nu_\phi(x)}$$
and
	$$\pi_{\phi,\ess}(f) := \bigoplus_{x\in\sigma_\ess(\phi)}f(x)\Id_{n_x},$$
with $n_x\in\IN_0 \cup \{\infty\}$ denoting the number of occurrences of $x$ in $\seq x$ and
$\Id_m\in M_m(\CC)$ the identity matrix for any $m\in\IN$, with $\Id_0 = 0$ by
definition. $\Id_{\infty}$ is thought of as  the identity on an infinite dimensional subspace.
To this decomposition of representations corresponds a decomposition of the associated order zero map of the analogous form
\begin{equation}\label{eq:cpcdec}
	\phi = \phi_i \hoplus \phi_{\ess},
\end{equation}
with obvious meaning of the symbols.

\begin{lemma}\label{lem:sigmai} Let $\phi,\psi:C(X)\to K$ be c.p.c. order zero
maps such that $\sigma_\ess(\phi) = \sigma_\ess(\psi)=\varnothing$. Then
$\phi\precsim\psi$ if and only if $\nu_\phi(x)\leq\nu_\psi(x)$ for all $x\in X$.
\end{lemma}
\begin{proof} If $\phi\precsim\psi$ and $\nu_\phi > \nu_\psi$, then there exists
at least one point $x\in\sigma_i(\phi)$ that either is not in $\sigma_i(\psi)$
or appears with a higher multiplicity than in the decomposition of $\psi_i$. In
both cases it is immediate to conclude that there cannot be a sequence that
witnesses $\phi\precsim\psi$ by a rank argument. Conversely, if
$\nu_\phi\leq\nu_\psi$ then there clearly is an isometry that conjugates $\pi_\phi$
inside $\pi_\psi$. To witness $\phi\precsim\psi$ it is enough to rescale every
basis vector by the appropriate factors coming from the eigenvalues of $h_\phi$
and $h_\psi$ and combine this with the above isometry, together with an approximate
unit for $K\subset B(H)$.
\end{proof}

The following result shows that the multiplicity of every point in the essential
spectrum of a c.p.c. order zero map $\phi:C(X)\to K$ is irrelevant since by
accumulation and the continuity of the functions in $C(X)$ they can be replaced
by nearby points in $\sigma_\ess(\phi)$.

\begin{lemma}\label{lem:presigmaess} Let $\phi:C(X)\to K$ be a non-zero c.p.c. order zero map with
$\sigma_i(\phi)=\varnothing$ and let $\seq y$ be a dense sequence of
$\sigma_\ess(\phi)$. Then $\pi_\phi$ is approximately unitarily equivalent to
$\rho := \bigoplus_{i\in\IN}\ev_{y_i}$.
\end{lemma}
\begin{proof} Fix a compatible metric $d$ on $\sigma_{\ess}(\phi)$ and $\eps >0$. 
By compactness we can find a bijection $\sigma = \sigma_{\eps} : \IN \to \IN$ such that 
$d(x_n , y_{\sigma(n)} ) < \eps$ for all $n$. Indeed, there exists a partition $C_1, \ldots , C_{N_{\eps}}$
of $X$ such that $\textup{diam}(C_i) < \eps$ and both $\{ x_n \} \cap C_i$ and 
$\{ y_n \} \cap C_i$ are infinite for all $i$. Choose any bijection between $P_i = \{ n \mid x_n \in C_i\}$ 
and $Q_i = \{ n \mid y_n \in C_i \}$; they piece together to the required bijection $\sigma$.
$\sigma_{\eps} $ defines a unitary $u_{\eps}$ such that  
$$
\| u_{\eps} \pi_{\phi} (f) u_{\eps}^* - \rho (f) \| \to 0 
$$
as $\eps \to 0$ for all $f \in C(X)$, so that $u_{1/n}$ gives the desired approximate unitary equivalence.
\end{proof}

\begin{lemma}\label{lem:sigmaess} Let $\phi,\psi:C(X)\to K$ be c.p.c. order zero
maps such that $\sigma_i(\phi) = \sigma_i(\psi)=\varnothing$. Then
$\phi\precsim\psi$ if and only if $\nu_\phi\leq\nu_\psi$.
\end{lemma}
\begin{proof} Assume that $\nu_\phi > \nu_\psi$ to conclude that there exists a
point $x\in\sigma_\ess(\phi)$ which is not in $\sigma_\ess(\psi)$. Since the
space $X$ is Hausdorff, there exists a neighborhood $U$ of $x$ which has empty
intersection with $\sigma_\ess(\psi)$. Now, by a similar argument as in Lemma \ref{lem:sigmai},
one can conclude that there cannot be a sequence that witnesses
$\phi\precsim\psi$. Conversely, if $\nu_\phi\leq\nu_\psi$ then there exists a
projection $E$ such that $E \pi_\psi E$ is approximately unitarily equivalent to
$\pi_\phi$, as a consequence of the previous lemma. It is then enough to rescale
every basis vector by the appropriate factors coming from the eigenvalues of
$h_\phi$ and $h_\psi$ and combine this with the unitary, together with an
approximate unit for $K\subset B(H)$ to witness $\phi\precsim\psi$.
\end{proof}

If $\seq x\subset\sigma_\ess(\phi)$ is a dense sequence in the essential
spectrum of a c.p.c. order zero map $\phi$, then $\{x_1,x_1,x_2,x_2,\ldots\}$ is
also dense and therefore $\phi\cong\phi_\ess\hoplus\phi$. Furthermore, the proof
of the above lemma, which relies on Lemma \ref{lem:presigmaess}, easily adapts
to the case of c.p.c. order zero maps $\phi$ and $\psi$ for which
$\sigma_i(\psi)=\sigma_\ess(\phi) = \varnothing$ and
$\sigma_i(\phi)\subset\sigma_\ess(\psi)$, at the price of having a sequence of
partial isometries rather than unitaries in general.

\begin{theorem}\label{thm:multiplicity} Let $\phi,\psi:C(X)\to K$ be c.p.c.
order zero maps. Then $\phi\precsim\psi$ if and only if $\nu_\phi\leq\nu_\psi$.
\end{theorem}
\begin{proof} With the above remarks in mind, and the decomposition of Equation \eqref{eq:cpcdec}, it is enough to decompose $\phi$ and $\psi$ in
	$$\phi = \phi_i\hoplus\phi_\ess\qquad\text{and}\qquad\psi = \psi_i\oplus\hoplus\psi_\ess.$$
Now $\phi_i$ can be decomposed further according to $\sigma(\psi)$ into
	$$\phi_i = \phi_{i,\ess}\hoplus\phi_{i,i},$$
where $\sigma(\phi_{i,\ess})\subset\sigma_\ess(\psi)$ and $\sigma(\phi_{i,i})\subset \sigma_i(\psi)$. By replacing $\psi$ with the equivalent map $\psi_\ess\hoplus\psi$ we then have the decompositions
	$$\phi = \phi_\ess\hoplus\phi_{i,\ess}\hoplus\phi_{i,i}\qquad\text{and}\qquad\psi = \psi_\ess\hoplus\psi_\ess\hoplus\psi_i.$$
It is now enough to apply Lemma
\ref{lem:sigmai} to $\phi_{i,i}$ and $\psi_i$, and Lemma \ref{lem:sigmaess} to the remaining direct summands, to get to the sought conclusion.
\end{proof}

Observe that any function $\nu$ from $X$ to $\mathbb{N}_0\cup\{\infty\}$ with compact support
can be split into the sum of two functions, $\nu_i$ and $\nu_\ess$, with
disjoint closed support, such that the former is supported by the isolated
points of $\supp\nu$ and $\nu_\ess$ on the rest of $\supp\nu$.

\begin{definition} Let $\tilde\IN_0:=\IN_0\cup\{\infty\}$. A Multiplicity function on $X$ is a function
$\nu : X \to \tilde\IN_0 $ whose support is closed and $\nu$ takes the value  $\infty$ on 
the accumulation points of its support.
We denote the set of multiplicity functions on $X$ by $\Mf(X)$.
\end{definition}

Observe that $\Mf(X)$ has a natural structure of partially ordered Abelian
monoid when equipped with the point-wise operation of addition, and partial order
inherited from $\tilde\IN_0$. In particular, every multiplicity function
$\nu_\phi$ of a c.p.c. order zero map $\phi:C(X)\to K$ is an element of
$\Mf(X)$. By Theorem \ref{thm:multiplicity}, every function in $\Mf(X)$ is
associated to a representation of $C(X)$ onto $K$ and hence to a c.p.c. order
zero map, which is unique up to Cuntz equivalence. If we insist on calling
$WW(C(X),\CC)$ the Cuntz homology of $X$, then we have the following

\begin{corollary} The Cuntz homology of $X$ is order isomorphic to the partially
ordered Abelian monoid $\Mf(X)$.
\end{corollary}

For the bivariant Cuntz semigroup $W(C(X),\CC)$ one sees that the only
representations of $C(X)$ involved are finite dimensional, for the positive
element $h_\phi$ of a c.p.c. order zero map $\phi:C(X)\to M_\infty$ is actually
a finite positive matrix in $M_n$ for some $n\in\IN$. Therefore, if we denote by
$\Mf_i(X)$ the subsemigroup of $\Mf(X)$ given by
	$$\Mf_i(X) := \{\nu\in \IN_0^X\ |\ |\supp\nu|<\infty\},$$
i.e. all the finitely supported multiplicity functions over $X$ with values in $\IN_0$, then $W(C(X),\CC) \cong \Mf_i(X)$ as partially ordered Abelian monoids. As the following result shows, the former can be regarded as the $\sup$-completion of the latter, that is $\Mf_i(X)$ is dense in $\Mf(X)$.

\begin{proposition} For every $\nu\in\Mf(X)$ there exists an increasing sequence $\seq\nu\subset\Mf_i(X)$ such that $\nu = \sup\nu_n$.
\end{proposition}
\begin{proof} Let $\seq x\subset\supp\nu_\ess$ be a dense sequence,
$Y\subset\supp\nu_i$ be such that $\nu_i(y) = \infty$ for any $y\in Y$, $Z :=
\supp\nu_i\smallsetminus Y$ and set
		$$\nu_n(x) : =\begin{cases}
		\nu_i(x) & x\in Z\\
		n & x\in\{x_1,\ldots,x_n\}\cup Y\\
		0	&\text{otherwise}
	\end{cases},\qquad\forall x\in X.$$
	Then $\nu_n\in\Mf_i(X)$ and $\nu_n\leq\nu$ for every $n\in\IN$. Suppose that
	$\mu\in\Mf(X)$ is such that $\nu_n\leq\mu$ for any $n\in\IN$, then
	$\supp\nu_n\subset\supp\mu$ for all $n\in\IN$ and by the closedness of the
	supports and the density of $\seq x$ in $\supp\nu_\ess$ it follows that
	$\supp\nu\subset\supp\mu$. This inclusion implies that $\nu_\ess\leq\mu_\ess$,
	while $\nu_i\leq\mu$ by construction of the sequence $\seq\nu$, whence $\nu =
	\nu_i+\nu_\ess\leq\mu$.
\end{proof}

Observe that any semigroup $\Mf(X)$ can be described as a quotient of countably supported functions on $X$ taking values in $\tilde\IN_0$. The equivalence relation is provided by checking that two functions agree in value on the isolated points and have the same closure of the accumulation points, where the functions take value $\infty$. If $\Lf(X)$ denotes the set of all such countably supported functions over $X$ and $\sim$ is said equivalence relation, then
	$$\Mf(X)\cong\Lf(X)/\sim.$$
An element $f$ of $\Lf(X)$ can be represented as a formal sum
	$$f = \sum_{k=1}^\infty n_k\delta_{x_k},$$
where $\seq x\subset X$ is a sequence of points from $X$ and $\{n_k\}_{k\in\IN}\subset\tilde\IN_0$ is such that each $n_k = \infty$ whenever $x_k$ is an accumulation point and $\delta_{x_k}$ is the function that takes value 1 on $x_k$ and $0$ everywhere else.

The image of $f$ under $\pi:\Lf(X)\to\Mf(X)$ can then be represented as the formal sum
	$$\pi(f) = \sum_{i=1}^\infty n_{k_i}\delta_{x_{k_i}} + \omega_C$$
where 
$C$ is the closure of all the accumulation points of the sequence $\seq x$, 
$$\omega_C(x) := \begin{cases}\infty & x\in C\\0 & x\not\in C.\end{cases}$$
and the sum is on the isolated points only.

The topology $\tau_X$ of the space $X$ can be recovered from the knowledge of $\Mf(X)$ as follows. For every two multiplicity functions $f,g\in {\rm Mf}(X)$ we say that $f$ is $\tau$-equivalent to $g$, in symbols $f\sim_\tau g$, if $\supp f=\supp g$. It is easy to see that the quotient
	$$\Tau(X):=\Mf(X)/\sim_\tau$$
is in a bijective correspondence with all the closed subsets of $X$ and hence with the topology $\tau_X$ on $X$. The set $\Tau(X)$ can also be identified with the class of those functions in $\Mf(X)$ whose range is in the set $\{0,\infty\}$. Such functions have the absorption property
	$$\omega + f = \omega$$
for any $f\in\Mf(X)$ such that $\supp f\subset\supp\omega$, and the stability property
	$$n\omega = \omega,\qquad\forall n\in\tilde\IN,$$
or just $\omega + \omega = \omega$.
All these properties characterize such elements which are in a one-to-one correspondence with closed subsets.

It is easy to see that  Cuntz homology as defined in this section provides a complete invariant for compact Hausdorff spaces. Formally we have the following

\begin{theorem}
Let $X$ and $Y$ be compact metrizable spaces. Then $WW(C(X), \CC)$ is isomorphic to $WW(C(Y), \CC)$ if and only if $X$ and $Y$ are homeomorphic. 
\end{theorem}
\begin{proof}
We only need to show that $\Mf(X)\cong\Mf(Y)$ via $\Phi : \Mf(X) \to \Mf(Y)$  implies the existence of a homeomorphism between $X$ and $Y$. 
To this end observe that an element in  $\Mf(X)$ is minimal and non-zero if and only if it is of the form $\delta_x$ for $x \in X$. Since an isomorphism preserves this property there exists a bijection $f: X \to Y$ such that $\Phi (\delta_x ) = \delta_{f(x)}$ for all $x \in X$. 

The elements $\omega  \in \Mf(X)$ which are additive idempotents i.e. satisfy $\omega + \omega = \omega$ are precisely the ones of the form 
$\omega_C$ for some closed subset $C \subset X$. This property is preserved under $\Phi$ so that we obtain a bijection $F$ between the closed sets of $X$ and the closed sets of $Y$. Since $\omega_C + \delta_x = \omega_C$ if and only if $x \in C$  we find $f(c) \in F(C)$ for all $c \in C$ i.e. $f(C) \subset F(C)$. Since the maps from $Y$ to $X$ and from the closed subsets of $Y$  to the closed subsets of $X$ induced by $\Phi^{-1}$ are $f^{-1}$ and $F^{-1}$ respectively we obtain $f(C)=F(C)$ and the same for the inverses. Thus, $f$ is a homeomorphism.
\end{proof}


\section{\label{sec:classification}Classification Results}

In this section we characterize the possible isomorphism between any two unital simple and stably finite C*-algebras in terms of the Bivariant Cuntz semigroup. In particular, this class contains the set of unital AF-algebras classified by Elliot and the set of unital AI-algebras, classified by Ciuperca and Elliott \cite{ceai}, among others.

Recall that for every element $\Phi\in WW(A,B)$, there exists a
c.p.c. order zero map $\phi:A\to B\otimes K$ such that
$[\phi\otimes\id_K]=\Phi$, and that we say that $\Phi\in WW(A,B)$ is invertible
if there exists a c.p.c. order zero map $\psi:B\otimes K\to A\otimes K$ such
that $\psi\circ\phi\sim\id_{A\otimes K}$ and $\phi\circ\psi\sim\id_{B\otimes
K}$, for any representative $\phi\in\Phi$. As in the case of K-theory, where the
unordered $K_0$-group is only capable of capturing stable isomorphisms between
AF-algebras, this notion of invertibility may present the same sort of limitations when is used to classify C*-algebras. We provide below an example that shows this limitation.

\begin{example}\label{ex:1st} Let $n, m > 0$ be natural numbers. By Example \ref{Ex:RecoverCuntz}, one computes that
$WW(M_n, M_m)=WW(M_m, M_n)=\mathbb{N}_0\cup\{\infty\}$. These semigroups contain invertible elements, namely $1\in\IN_0\cup\{\infty\}$. However, $M_n$ and $M_m$ are stably isomorphic, i.e. $M_n\otimes K\cong M_m\otimes K$ for any $m,n\in\IN_0$, but isomorphic only if $n=m$.
\end{example}

This simple example shows that, in order to capture isomorphism, a stricter
notion of invertibility is required. As a guiding principle, we have Elliott's
classification theory of AF-algebras through their \emph{dimension groups}, that
is the collection of the \emph{ordered} $K_0$-groups, its scale, and the class
of the unit of the algebra in the unital case.

With the following result we establish that, in the unital and stably finite case, a pair of c.p.c. order zero maps which are invertible up to Cuntz
equivalence are Cuntz equivalent to their support $*$-homomorphisms. Recall, before it, that for any simple and stably finite \Cs-algebra $A$
, one has an embedding of the Murray-von Neumann
semigroup $V(A)$ into $\Cu(A)$ (cf. \cite{apt2009})

\begin{proposition}\label{prop:cpctohom} Let $A,B$ be unital and finite \Cs-algebras. If $\phi:A\to B$ and $\psi:B\to A$ are two c.p.c.
order zero maps such that $\psi\circ\phi\sim \id_A$ and
$\phi\circ\psi\sim\id_B$, then there are unital $*$-homomorphisms $\pi_\phi:A\to
B$ and $\pi_\psi:B\to A$ such that
	\begin{enumerate}[{\rm (i)}]
		\item $[\pi_\phi] = [\phi]$ and $[\pi_\psi] = [\psi]$;
		\item $\pi_\psi\circ\pi_\phi \sim \id_A$ and
		$\pi_\phi\circ\pi_\psi\sim\id_B$.
	\end{enumerate}
\end{proposition}
\begin{proof} By \sth, we find positive elements $h_\phi,h_\psi$ and
$*$-homo\-mor\-phisms $\pi_\phi,\pi_\psi$ such that $\phi = h_\phi\pi_\phi$ and
$\psi = h_\psi\pi_\psi$. Evaluating on the unit of $A$ and $B$ respectively, we
get
		$$h_\psi^{\frac12}\pi_\psi(h_\phi)h_\psi^{\frac12}\sim_\Cu 1_A
		\qquad\text{and}\qquad
		h_\phi^{\frac12}\pi_\phi(h_\psi)h_\phi^{\frac12}\sim_\Cu 1_B.$$
	From the first relation on the left we get, by definition of $\cue$, the
	existence of a sequence $\seq x\subset A$ such that
	$x_nh_\psi^{\frac12}\pi_\psi(h_\phi)h_\psi^{\frac12}x_n^*$ converges to $1_A$
	in norm, and therefore
	$x_nh_\psi^{\frac12}\pi_\psi(h_\phi)h_\psi^{\frac12}x_n^*$ is eventually
	invertible. Hence, there exist $c_n\in A$ such that
	$x_nh_\psi^{\frac12}\pi_\psi(h_\phi)h_\psi^{\frac12}x_n^*c_n = 1_A$ for
	sufficiently large values of $n$, which shows that $x_n$ is right invertible.
	Since $A$ is finite, it follows that the sequence $\seq x$ is
	eventually invertible, and therefore
	$h_\psi^{\frac12}\pi_\psi(h_\phi)h_\psi^{\frac12}x_n^*c_nx_n = 1_A$, which
	shows that $h_\psi$ is also invertible for the same argument as before.
	Similarly, one also deduce the invertibility of $h_\phi$, so that $\pi_\phi$
	and $\pi_\psi$ satisfy (i) and (ii).
	
	Now set $p=\pi_\phi(1_A)$ and $q =
	\pi_\psi(1_B)$. Since $\pi_\psi(p)\sim_\Cu 1_A$ and $\pi_\phi(q)\sim_\Cu 1_B$,
	finiteness of $A$ and $B$ implies $\pi_\phi(q) = 1_B$ and $\pi_\psi(p)
	= 1_A$. Now $1_A-q$ is a positive element in $A$, but
		$$\pi_\phi(1_A-q) = p - 1_B \leq 0,$$
	which is possible only if $p= 1_B$. Similarly, one finds that $q=1_A$, and
	therefore $\pi_\phi$ and $\pi_\psi$ are unital.
\end{proof}

In order to lift an invertible element in the bivariant Cuntz semigroup, it
suffices to show the existence of representatives which are $*$-homomorphisms,
but in a strict sense. These considerations motivate the following definition.

\begin{definition}[Strict invertibility] An element $\Phi\in WW(A,B)$ is
strictly invertible if there exist c.p.c. order zero maps $\phi:A\to B$ and
$\psi:B\to A$ such that
	\begin{enumerate}[i.]
		\item $[\phi\otimes\id_K] = \Phi$;
		\item $\psi\circ\phi \sim \id_A$ and $\phi\circ\psi\sim\id_B$.
	\end{enumerate}
\end{definition}

Observe that every $*$-isomorphism between two \Cs-algebras $A$ and $B$ induces
a strictly invertible element in $WW(A,B)$. Hence, if there are no strictly
invertible elements in $WW(A,B)$, then $A$ and $B$ cannot be isomorphic.

\begin{definition}[Strict WW-equivalence] Two \Cs-algebras $A$ and $B$ are
strictly $WW$-equivalent if there exists a strictly invertible element in
$WW(A,B)$.
\end{definition}

\begin{remark} The notions of strictly invertible elements and strict
$WW$-equivalence can be reformulated for the bivariant semigroup $W$ as well in
the same formal terms.
\end{remark}

Observe that any c.p.c. order zero map $\phi:A\to B$ induces an element in
$WW(A,B)$ through the class of its ampliation, viz. $[\phi\otimes\id_K]$. To
make tangible contact with the current theory of classification, we also
introduce the next definition.

\begin{definition}[Scale of $WW$] The scale $\Sigma(WW(A,B))$ of the bivariant
Cuntz semigroup $WW(A,B)$ is the set of all classes of c.p.c. order zero maps
that arise from c.p.c. order zero maps from $A$ to $B$ through
$\infty$-ampliation, i.e. the set
		$$\Sigma(WW(A,B)) = \{[\phi\otimes\id_K]\in WW(A,B)\ |\ \phi:A\to B\text{
		c.p.c. order zero}\}.$$
\end{definition}

The above definition of scale for the bivariant Cuntz semigroup can be used to
characterize the strictly invertible elements in $WW(A,B)$ as those invertible
element in the scale of $WW(A,B)$, which have an inverse in $\Sigma(WW(B,A))$.

\begin{example} Let $B$ be a \Cs-algebra. Since any c.p.c. order zero map $\phi
: \CC \to B$ is generated by a positive element of $B$, i.e.
		$$\phi(z) = zh_\phi,\qquad\forall z\in\CC,$$
	for some positive element $h_\phi\in B^+$, one can identify the $\Sigma(WW(\CC,B))$ with the Cuntz equivalence classes of the elements of $B$
	embedded in $B\otimes K$ through a minimal projection $e$ of $K$. Apart from a
	suprema-completion, $\Sigma(WW(\CC,B))$ coincides with the notion of scale for
	the ordinary Cuntz semigroup introduced in \cite{ptww}.
\end{example}

We will now give a proof by examples for the classification of UHF algebras,
starting by revisiting Example \ref{ex:1st}.

\begin{example} Let $0 < n \leq m$ be natural numbers, and consider the matrix
algebras $M_n$ and $M_m$. We claim that there is a strictly invertible element
in $WW(M_n, M_m)$ if and only if $n = m$. One
direction is obvious; therefore, suppose that $\Phi\in WW(M_n,M_m)$ is a strictly
invertible element. Then, there are c.p.c. order zero maps $\phi : M_n \to M_m$
and $\psi : M_m \to M_n$ such that $\Phi = [\phi\otimes\id_K]$ and
$\psi\circ\phi\sim 1_{M_n}$, $\phi\circ\psi\sim 1_{M_m}$. By Proposition
\ref{prop:cpctohom}, we find unital $*$-homomorphisms $\pi_\phi:M_n\to M_m$ and
$\pi_\psi:M_m\to M_n$, but such a $\pi_\psi$ can only exist if $m=n$. \emph{En
passant} we observe that, under these circumstances, both $\pi_\phi$ and
$\pi_\psi$ are surjective and hence $*$-isomorphisms.
\end{example}

\begin{example}Let $A$ and $B$ be UHF algebras. As in the above example, one direction is trivial, so let us show the converse. To this end, following the above proof with $A$ and $B$ instead of
matrix algebras, one gets unital injective $*$-homomorphisms $\pi_\phi:A\to
B$ and $\pi_\psi:B\to A$. These only exist if $A$ and $B$ have the same
supernatural number by Glimm's classification result of UHF algebras
\cite{glimm}; hence, the desired implication follows. 
\end{example}

In particular, the above example shows the following.
\begin{proposition}\label{prop:ClaUHF} Two UHF algebras $A$ and $B$ are isomorphic if and only if there
is a strictly invertible element in $WW(A,B)$.
\end{proposition}

We finish this section extending Proposition \ref{prop:ClaUHF} to all unital stably finite
\Cs-algebras. Recall that it includes all $AF$-algebras and $AI$-algebras, as mentioned before. We first
give a lemma and recall Elliott's intertwining argument \cite{elliott10}.

\begin{lemma}\label{lem:cuau} Let $A$ and $B$ be unital \Cs-algebras, $B$ finite. Two unital $*$-homo\-mor\-phisms $\pi_1,\pi_2:A\to B$ are Cuntz
equivalent if and only if they are approximately unitarily equivalent.
\end{lemma}
\begin{proof} One implication is obvious, so assume that $\pi_1\sim\pi_2$. By assumptions, there
exists a sequence $\seq x\subset B$ such that
$\norm{x_n^*\pi_1(a)x_n-\pi_2(a)}\to0$ for any $a\in A$, which, in particular, implies that
	$$x_n^*x_n\to 1.$$
Now, by finiteness of $B$, one obtains that $x_n$ is eventually
invertible. Hence, without loss of generality and forgetting about the
first few non-invertible elements, if any, one can replace $x_n$ by the
unitaries $u_n$ coming from the polar decomposition $x_n = u_n|x_n|$ of each invertible $x_n$. The sequence $\seq u$ then witnesses the sought approximate unitary equivalence between $\pi_1$ and $\pi_2$.
\end{proof}

\begin{theorem}[Elliott \cite{elliott10}]\label{thm:intarg} Let $A$ and $B$ be
separable, unital \Cs-algebras. If there are unital $*$-homomorphisms
$\pi_1:A\to B$ and $\pi_2:B\to A$ such that $\pi_2\circ\pi_1\au\id_A$ and
$\pi_1\circ\pi_2\au\id_B$, then $A$ and $B$ are isomorphic.
\end{theorem}
Combining Proposition \ref{prop:cpctohom} with Lemma \ref{lem:cuau} and
Theorem \ref{thm:intarg}, we get to the following classification result for
unital and stably finite \Cs-algebras in the bivariant Cuntz Semigroup setting.

\begin{theorem}\label{thm:classification} Let $A$ and $B$ be unital, stably
finite \Cs-algebras. There is an isomorphism between $A$ and $B$ if and only if
there exists a strictly invertible element in $WW(A,B)$.
\end{theorem}
\begin{proof} It is clear that any isomorphism between $A$ and $B$ gives a
strictly invertible element. To prove the converse, assume that
$\Phi\in WW(A,B)$ is a strictly invertible element and that $\phi$ is a
representative of $\Phi$. By Proposition \ref{prop:cpctohom}, one finds unital
$*$-homomorphisms $\pi_1:A\to B$ and $\pi_2:B\to A$ such that
$\pi_2\circ\pi_1\sim\id_A$ and $\pi_1\circ\pi_2\sim\id_B$. Then, Cuntz equivalence is replaced by approximately unitary
equivalence by Lemma
\ref{lem:cuau}, so one gets that $A$ is isomorphic to $B$ by Theorem
\ref{thm:intarg}.
\end{proof}

The well-known Kirchberg-Phillips classification (\cite{KP002,KP00}) says that Kirchberg algebras, i.e. simple separable nuclear and purely infinite \Cs-algebras are isomorphic if and only if they are KK-equivalent.
Our classification theorem \ref{thm:classification} can be regarded as an analogue of this result, in fact goes beyond it in some sense since we do not require any conditions beyond stable finiteness such as nuclearity or simplicity. Moreover, as our examples show any two Kirchberg algebras are Cuntz equivalent, our bivariant Cuntz semigroup appears as exclusively geared to algebras with traces.

	\bibliographystyle{plain}
	\bibliography{refs}

\begin{thebibliography}{10}

\bibitem{ABP11}
R.~Antoine, J.~Bosa, and F.~Perera.
\newblock {Completions of monoids with applications to the Cuntz semigroup}.
\newblock {\em {Int. J. of Math.}}, 22(6):837--861, 2011.

\bibitem{ABPP14}
R.~Antoine, J.~Bosa, and F.~Perera.
\newblock {The Cuntz semigroup of continuous fields}.
\newblock {\em {Indiana Un. Math. J.}}, 62(4):1105--1131, 2013.

\bibitem{apt2014}
R.~Antoine, F.~Perera, and H.~Thiel.
\newblock {Tensor products and regularity properties of Cuntz semigroups}.
\newblock {\em {to appear in Memoirs Am. Math. Soc.}}, {oct} 2014.

\bibitem{apt2009}
P.~Ara, F.~Perera, and A.~S. Toms.
\newblock {$K$-theory for operator algebras. Classification of $C^*$-algebras.}
\newblock In {\em {Aspects of Operator Algebras and Application}}, volume 534
  of {\em {Contem. Math.}} {AMS}, 2011.

\bibitem{btzop}
J.~Bosa, G.~Tornetta, and J.~Zacharias.
\newblock {Open projections and suprema in the Cuntz semigroup}.
\newblock {\em {arXiv:1411.5619}}, 2015.

\bibitem{bpt08}
N.~P. Brown and A.~Ciuperca.
\newblock {Isomorphism of Hilbert modules over stably finite C*-algebras}.
\newblock {\em {J. of Funct. Anal. }}, 257(1):332--339, 2009.

\bibitem{bc}
N.~P. Brown, F.~Perera, and A.~S. Toms.
\newblock {The Cuntz semigroup, the Elliott conjecture and dimension functions
  on C*-algebras}.
\newblock {\em {J. Reine Angew. Math.}}, 621:191--211, 2008.

\bibitem{ce}
M.~Choi and E.~G. Effros.
\newblock Nuclear {$C\sp*$}-algebras and the approximation property.
\newblock {\em Amer. J. Math.}, 100(1):61--79, 1978.

\bibitem{ceai}
A.~Ciuperca and G.~A. Elliott.
\newblock A remark on invariants for {$C^*$}-algebras of stable rank one.
\newblock {\em Int. Math. Res. Not. IMRN}, 5:Art. ID rnm 158, 33, 2008.

\bibitem{cei2008}
K.~T. Coward, G.~A. Elliott, and C.~Ivanescu.
\newblock {The Cuntz semigroup as an invariant for C*-algebras.}
\newblock {\em {J. Reine Angew. Math. }}, 623:161--193, 2008.

\bibitem{cuntz78}
J.~Cuntz.
\newblock {Dimension functions on simple C*-algebras.}
\newblock {\em {Math. Ann.}}, 233:145--154, 1978.

\bibitem{dadarlat}
M.~Dadarlat.
\newblock {Continuous fields of C*-algebras over finite dimensional spaces}.
\newblock {\em {Adv. in Math.}}, 222(5):1850--1881, 2009.

\bibitem{elliott10}
G.~A. Elliott.
\newblock {Towards a theory of classification }.
\newblock {\em {Adv. in Math. }}, 223(1):30--48, 2010.

\bibitem{ERS11}
G.~A. Elliott, L.~Robert, and L.~Santiago.
\newblock {The cone of lower semicontinuous traces on a C*-algebra}.
\newblock {\em {Amer. J. Math}}, 133(4):969--1005, 2011.

\bibitem{glimm}
J.~Glimm.
\newblock {On a certain class of operator algebras}.
\newblock {\em {Trans. Amer. Math. Soc.}}, 95(2):318{--}340, 1960.

\bibitem{handelman}
D.~Handelman.
\newblock {Homomorphisms of $C^*$-algebras to finite $AW^*$-algebras.}
\newblock {\em {Michigan Math. J.}}, 28(2):229--240, 1981.

\bibitem{KP002}
E.~Kirchberg and N.~Phillips.
\newblock {Embedding of continuous fields of C*-algebras in the Cuntz
  $\mathcal{O}_2$}.
\newblock {\em {J. Reine Angew. Math.}}, 525:55--94, 2000.

\bibitem{KP00}
E.~Kirchberg and N.~Phillips.
\newblock {Embedding of exact C*-algebras in the Cuntz algebra
  $\mathcal{O}_2$}.
\newblock {\em {J. Reine Angew. Math.}}, 525:17--532, 2000.

\bibitem{ort}
E.~Ortega, M.~R{\o}rdam, and H.~Thiel.
\newblock {The Cuntz semigroup and comparison of open projections }.
\newblock {\em {J. of Funct. Anal. }}, 260(12):3474--3493, 2011.

\bibitem{pereratoms}
F.~Perera and A.~S. Toms.
\newblock {Recasting the Elliott conjecture}.
\newblock {\em {Math. Ann.}}, 338:669--702, 2007.

\bibitem{ptww}
F.~Perera, A.~S. Toms, S.~White, and W.~Winter.
\newblock {The Cuntz semigroup and stability of close \Cs-algebras}.
\newblock {\em {Analysis \& PDE}}, pages 929--952, 2014.

\bibitem{robert}
L.~Robert.
\newblock {Classification of inductive limits of 1-dimensional NCCW complexes}.
\newblock {\em {Adv. in Math.}}, 231(5):2802--2836, 2012.

\bibitem{rordam}
M.~R{\o}rdam and E.~St{\o}rmer.
\newblock {\em Classification of nuclear {$C^*$}-algebras. {E}ntropy in
  operator algebras}, volume 126 of {\em Encyclopaedia of Mathematical
  Sciences}.
\newblock Springer-Verlag, Berlin, 2002.
\newblock Operator Algebras and Non-commutative Geometry, 7.

\bibitem{toms}
A.~Toms.
\newblock {On the classification problem for nuclear \Cs-algebras}.
\newblock {\em {Ann. of Math.}}, 167(2):1059--1074, 2008.

\bibitem{tw}
A.~S. Toms and W.~Winter.
\newblock {Strongly self-absorbing \Cs-algebras}.
\newblock {\em {Trans. Amer. Math. Soc.}}, 359(8):3999--4029, 2007.

\bibitem{wz2009}
W.~Winter and J.~Zacharias.
\newblock {Completely positive maps of order zero}.
\newblock {\em {Munster J. of Math.}}, 2:311--324, 2009.

\end{thebibliography}
	
\end{document}